\documentclass[a4paper,11pt]{article}
\usepackage{amsmath,amsfonts,epsfig}

\catcode`@=11 \@addtoreset{equation}{section}
\renewcommand\theequation{\thesection.\@arabic\c@equation}
\catcode`@=12

\usepackage{diagbox}
\usepackage{multirow}
\usepackage{booktabs}
\usepackage{longtable}

\usepackage{dutchcal}
\def\theequation{\thesection.\arabic{equation}}
\usepackage{amsmath,amsfonts,amssymb}
\usepackage{booktabs}
\usepackage{amsmath,amsfonts,epsfig}
\usepackage{graphics}
\usepackage{supertabular}
\usepackage{amsthm}
\usepackage{mathtools}
\usepackage{amsfonts}
\usepackage{cite}
\usepackage{graphicx}
\usepackage{epstopdf}
\usepackage[numbers]{natbib}
\usepackage{float,epsfig, floatflt,here}
\usepackage{color}
\usepackage[colorlinks=true,citecolor=blue,linkcolor=blue]{hyperref}
\usepackage{fancyhdr}
\usepackage{etoolbox}
\usepackage{dsfont }
\usepackage{accents}

\usepackage{mathrsfs}
\usepackage{ upgreek }

\newcommand{\X}{\mathrm{{\cal X}}}

\newcommand{\tnorm}[1]{{\left\vert\kern-0.25ex\left\vert\kern-0.25ex\left\vert #1\right\vert\kern-0.25ex\right\vert\kern-0.25ex\right\vert}}

\newcommand{\vertiii}[1]{{\left\vert\kern-0.25ex\left\vert\kern-0.25ex\left\vert #1
		\right\vert\kern-0.25ex\right\vert\kern-0.25ex\right\vert}}
\newcommand{\noN}{\nonumber}

\newtheorem{thm}{Theorem}[section]

\newtheorem{lemma}{Lemma}[section]

\newtheorem{rem}{Remark}[section]

\usepackage{algorithm}
\usepackage{a4wide,enumerate,xcolor,graphicx}
\usepackage{subcaption}
\usepackage{float}

\let\pa\partial
\let\na\nabla
\let\eps\varepsilon
\newcommand{\diver}{\operatorname{div}}

\renewcommand{\theequation}{A.\arabic{equation}}
\numberwithin{equation}{section}
\numberwithin{figure}{section}
\numberwithin{table}{section}
\begin{document}
	\title{ \textbf{{\em Error Analysis of a Fully Discrete Scheme for The Cahn--Hilliard Cross-Diffusion Model in Lymphangiogenesis}}}
	\author{ Boyi Wang \thanks{Department of Mathematics, Pennsylvania State University, 54 McAllister Street, Pennsylvania, USA (buw176@psu.edu)} \and Naresh Kumar \thanks{Department of Mathematics and Computing, Dr B R Ambedkar National Institute of Technology Jalandhar, India (nareshk@nitj.ac.in)} 
		\and Jinyun Yuan\thanks{School of Computer Science and Technology, Dongguan University of Technology, Dongguan, China (yuanjy@gmail.com)} }
	\date{}
	\maketitle
	
\begin{abstract}
This paper introduces a stabilized finite element scheme for the Cahn--Hilliard cross-diffusion model, which is characterized by strongly coupled mobilities, nonlinear diffusion, and complex cross-diffusion terms. These features pose significant analytical and computational challenges, particularly due to the destabilizing effects of cross-diffusion and the absence of standard structural properties. To address these issues, we establish discrete energy stability and prove the existence of a finite element solution for the proposed scheme. A key contribution of this work is the derivation of rigorous error estimates, utilizing the novel $L^{\frac{4}{3}}(0,T; L^{\frac{6}{5}}(\Omega))$ norm for the chemical potential. This enables a comprehensive convergence analysis, where we derive error estimates in the $L^{\infty}(H^1(\Omega))$ and $L^{\infty}(L^2(\Omega))$ norms, and establish convergence of the numerical solution in the $L^{\frac{4}{3}}(0,T; W^{1,\frac{6}{5}}(\Omega))$ norm. Furthermore, the convergence analysis relies on a uniform bound of the form $\sum_{k=0}^n\tau\|\nabla(\cdot)\|_{L^{\frac{6}{5}}}^{\frac{4}{3}}$ to control the chemical potentials, marking a clear departure from the classical $\sum_{k=0}^n\tau\|\nabla(\cdot)\|_{L^{2}}^{2}$ estimate commonly used in Cahn--Hilliard-type models. Our approach builds upon and extends existing frameworks, effectively addressing challenges posed by cross-diffusion effects and the lack of uniform estimates. Numerical experiments validate the theoretical results and demonstrate the scheme’s ability to capture phase separation dynamics consistent with the Cahn--Hilliard equation. 
\end{abstract}



\section{\normalsize Introduction} 
\vspace{-0.3cm}
\subsection{Modeling background}	
\vspace{-0.3cm}

The Cahn--Hilliard cross-diffusion model has emerged as an important mathematical framework for simulating lymphangiogenesis—the formation of new lymphatic vessels from pre-existing ones. This biological process plays a vital role in various physiological and pathological contexts, including tissue regeneration, immune responses, and cancer metastasis.  Recently, a Cahn--Hilliard (CH) cross-diffusion model with an energy-based structure has been proposed to capture morphological pattern formation \cite{JuLi25, JuLi24, JuWa23}. Derived via an energy variational approach, the model is formulated as
\begin{equation}
	\partial_t \textbf{u} = \nabla \cdot \big(\textbf{M}(\textbf{u}) \nabla \delta_{u} E\big),
\end{equation}
where \( E(\textbf{u}) \) is the energy functional, and \( \textbf{M}(\textbf{u}) \) denotes a cross-diffusion mobility matrix. 

The origins of this model trace back to earlier studies \cite{Hug41, RoFo08}, where it was used to describe network formation phenomena, particularly the sprouting of lymphatic vessels in biological systems \cite{RoFo08, RBS06}. A key property of the CH framework—shared with classical phase-field models \cite{ElGa96}—is its energy dissipation structure. This property is critical for ensuring the stability of numerical approximations, proving analytical results, and deriving error estimates.  In the case of phase-field models, these features have been thoroughly analyzed, resulting in a wide range of efficient numerical schemes supported by strong theoretical guarantees \cite{CWWW19, CE1992, DWW16, DuJu18, FTY13, FTY24, LJWZ22, SXY19, SXY18, TWY22, TXY22}. However, extending the Cahn--Hilliard framework to include a general cross-diffusion matrix \( \textbf{M}(\textbf{u}) \) introduces substantial mathematical and computational complexities. Unlike standard diffusion operators with scalar or diagonal mobilities, the cross-diffusion matrix couples multiple components of \( \textbf{u} \), significantly affecting solution regularity, stability analysis, and well-posedness. Classical energy estimates often fail in this setting due to intricate inter-variable interactions, demanding more sophisticated analytical techniques. This study addresses these challenges by conducting a rigorous error analysis based on novel norm estimates that are specifically adapted to the nonlinear structure of the cross-diffusion system.
\subsection{Problem description}	
In this paper, we investigate error estimates and study the convergence of a numerical scheme for the following Cahn--Hilliard cross-diffusion model
	\begin{equation}\label{1.mu}
	\begin{cases}
	\pa_t\phi = \diver(\na\mu-c\na h_c(\phi,c)), \\
	\pa_t c = -\diver(c(\na\mu-c\na h_c(\phi,c))) + \diver(g(c)\nabla h_c(\phi,c)),\\
	\mu = -\Delta\phi + \eps^{-2}f(\phi) + h_\phi(\phi,c)
	\end{cases}
	\end{equation}
for $x\in \Omega,\ 0<t\le T$, with periodic or homogeneous Neumann boundary conditions and the initial data
\begin{align}\label{1.ic}
	\phi(x,0) = \phi_0(x),\quad c(x,0)=c_0(x), \quad x\in \Omega.
\end{align}
Here, $\Omega$ is the bounded domain of $\mathbb{R}^d$ ($d=1,2,3$) with the smooth boundary $\partial\Omega$ and $T>0$ is the final time. The nonlinear terms $f$ and $h$ are from the energy density functions, and it can be shown that the system possesses the energy dissipation law 
\begin{align}
	\frac{\text{d}E}{\text{d}t}\leq 0,\quad \mbox{where } E = \int_{\Omega}\bigg(\frac{1}{2}|\na\phi|^2 + \eps^{-2}F(\phi)+h(\phi,c)\bigg)dx. \notag
\end{align}
The nonlinear terms are assumed to be the sum of the classical potential function $\frac{1}{\eps^2}F(\phi)$ and the nutrient energy $h(\phi,c)$.
We focus on the estimates for the model \eqref{1.mu} with the nutrient energy density $h(\phi,c)$ given by \cite{GKT22}  
\begin{align}
	h(\phi,c)=c^2/2-c\phi\label{he-1},
\end{align}
and $g(c) = 1$. In this case, since $h_c$ is linear with respect to $\phi$ and $c$, we may simply use the following homogenous bounday condition 
\begin{align}\label{1.bn}\frac{\partial \phi}{\partial n} =\frac{\partial c}{\partial n}=\frac{\partial \mu}{\partial n} = 0, \quad x\in \partial\Omega,\ 0\le t\le T.
\end{align} 
In this paper, we assume $f'$ is bounded with $f(\phi) = F'(\phi)$ for some regular potential $F$ and analyze an approximate model by applying a truncation technique to the potential function $F$. In addition, we assume that there exists the positive constants $K_i>\varepsilon^2,\, i=1,2$ and $L\ge 1$ such that
\begin{align}
	&F(\phi)\geq K_1\phi^2-K_2,\, \forall\; \phi\in \mathbb{R^d},\label{apf-1}\\
	&\vert f'(\phi)\vert\leq L,\,\forall\; \phi\in \mathbb{R^d}.\label{apf-2}
\end{align}
For example, we may use the double-well potential with truncation
\begin{equation}
	F(s)=\left\{
	\begin{array}{lr}
		\begin{aligned}
			& \frac{3M^2-1}{2}(s-M)^2+(M^3-M)(s-M)+\frac14(M^2-1)^2,&&  s\geq M,\\
			& \frac14(s^2-1)^2, && s\in[-M,M],\\
			& \frac{3M^2-1}{2}(s+M)^2-(M^3-M)(s+M)+\frac14(M^2-1)^2,&&  s\leq -M.\label{f-trc}
		\end{aligned}
	\end{array}\right.
\end{equation}
or the logarithmic potential in \cite{GKT22, JuWa23}.

From a numerical point of view, designing stable and accurate finite element schemes for these models is quite challenging. Cross-diffusion terms can easily cause instabilities if not treated properly, requiring effective stabilization methods and a careful choice of function spaces. Moreover, proving the convergence of numerical solutions demands careful and precise estimates that take into account the complex structure of  \( \textbf{M}(\textbf{u}) \). While structure-preserving numerical methods have been explored for similar problems \cite{EMP21, JuLi24, JuLi25, JuWa23},   the error analysis and rigorous convergence proofs for CH cross-diffusion models remain relatively unexplored.  
This paper aims to address these challenges by developing a stabilized finite-element scheme for a specific CH cross-diffusion model. We provide rigorous error estimates and establish the convergence of numerical solutions, leveraging the underlying energy structure of the model. Our analysis not only deepens the theoretical understanding of cross-diffusion systems but also provides practical insights for designing reliable numerical methods in this setting.  
\vspace{-0.5cm}

\subsection{Main Contributions and Overview}
\vspace{-0.3cm}

This work presents a stabilized finite element scheme for a Cahn--Hilliard-type cross-diffusion model characterized by strongly coupled mobilities, nonlinear diffusion, and fourth-order coupling terms. These features pose substantial analytical and numerical challenges, primarily due to the destabilizing nature of cross-diffusion and the absence of standard structural properties such as maximum principles or classical energy dissipation. While structure-preserving numerical methods have been proposed for related models~\cite{EMP21, JuLi24, JuLi25, JuWa23}, rigorous convergence analysis and error estimates for cross-diffusion systems remain scarce.

To address these challenges, our main contributions are as follows:
\begin{enumerate}[(i)]
	\item We develop a fully discrete, stabilized finite element method based on backward Euler time discretization. The existence of discrete solutions is established using Brouwer’s fixed-point theorem, following the approach in~\cite[Section 4.4]{GKT22}.
	
	\item By employing the model’s variational energy structure, we derive rigorous error estimates and prove error bounds of the numerical solution (see Theorem~4.1). This enhances the theoretical understanding of cross-diffusion systems and supports the reliability of the proposed method.
	
	\item A central novelty of this work is the convergence analysis, which is based on the estimate $\sum_{k=0}^n\tau\|\nabla(\cdot)\|_{L^{\frac{6}{5}}}^{\frac{4}{3}}$, used to control the chemical potentials. This departs from the classical estimate $\sum_{k=0}^n\tau\|\nabla(\cdot)\|_{L^{2}}^{2}$ typically employed in Cahn--Hilliard-type equations. The analysis also reveals a structural connection between the cross-diffusion model and the standard Cahn--Hilliard equation, highlighting the robustness and effectiveness of our numerical scheme.
	
	\item Our analysis generalizes the framework introduced in~\cite{CaSh18} for standard Cahn--Hilliard models. A major difficulty in our setting arises from the cross-diffusion terms, which preclude uniform-in-time $L^2$ bounds for $\nabla \mu^{n+1}$ and complicate control of inverse Laplacian error norms.
	
	\item Building on the foundational ideas of~\cite{CE1992}, we derive a novel a priori estimate for the discrete chemical potential in the space $L^{\frac{4}{3}}(0,T; W^{1,\frac{6}{5}}(\Omega))$. This compensates for the lack of a standard $L^2(0,T; H^1(\Omega))$ estimate and facilitates convergence analysis under mild regularity assumptions on the initial data $(\phi_0, c_0)$.
	
	\item We perform two-dimensional numerical simulations that validate the theoretical results and demonstrate the scheme’s capability to capture phase separation dynamics consistent with classical Cahn--Hilliard behavior.
\end{enumerate}
Overall, this work advances the analytical and computational framework for cross-diffusion systems by providing a mathematically rigorous and numerically stable finite element approach. The techniques and insights developed here lay the groundwork for further investigation of complex multiphase systems governed by nonlinear and degenerate dynamics.

The rest of this paper is organized as follows. One fully discrete finite-element numerical scheme is proposed and the energy dissipation stability is proven in Section \ref{sec-en}. The existence of the solution to the numerical scheme is proven in Section \ref{e-ns}. The rigorous error analysis with error estimates is carried out in Section \ref{sec-er} and the convergence of the numerical solution to the solution of the continuous problem is proven in Section \ref{sec-can}. Some numeric results are given in Section \ref{sec-nr}, and conclusion is stated in Section \ref{sec-con}. In general, the second-order derivative of the energy density $F$ unbounded, and $g(c)$ may lead to a degenerate diffusion matrix, adding further complexity. However, structure-preserving schemes can be proposed \cite{JuWa23}. We left the numerical analysis of these cases as a topic for future work.

\section{Finite element scheme and energy estimate}\label{sec-en}
\vspace{-0.45cm}

To advance toward the finite element algorithm, we now outline the finite element discretization over a computational domain $\Omega \subset \mathbb{R}^d\;(d=2,3)$ with boundary $\partial \Omega$.

\textbf{Finite Element Space Construction.}  
Let $\mathcal{T}_h$ be a conforming triangulation of $\Omega$ into non-overlapping triangles $K$ such that:
\begin{itemize}
	\item [(i)] $\overline{\Omega} = \bigcup_{K \in \mathcal{T}_h} K$,
	\item [(ii)] for any $K_1, K_2 \in \mathcal{T}_h$, either $K_1 \cap K_2 = \emptyset$, or their intersection is a common vertex or edge (i.e., no hanging nodes),
	\item [(iii)] the mesh size is defined by $h = \max_{K \in \mathcal{T}_h} h_K$, where $h_K$ is the diameter of triangle $K$.
\end{itemize}
\vspace{-0.15cm}


Based on this triangulation, we define the conforming finite element space ${\cal X}_h$ as
\begin{eqnarray*}
\X_h :=\{\psi_h\in C(\bar\Omega):\;\;\psi_h|_K\in P_l(K), \;\;\forall\; K\in T_h\}\subset H^1(\Omega).
\end{eqnarray*}
where $P_l(K)$ denotes the space of polynomials of degree at most $l$ on $K$.

Next, we assume the finite element space $\X_h$ satisfies the following approximation property \cite{PG}: for all $v \in H^s(\Omega) \cap H_0^1(\Omega)$ with $1 \leq s \leq l$, there exists a constant $C > 0$ independent of $h$ such that
\begin{equation*}
	\inf_{\psi \in \X_h} \left\{ \|v - \psi\| + h \|\nabla(v - \psi)\| \right\} \leq C h^s \|v\|_{H^s(\Omega)}.
\end{equation*}


Next, we define  $P_{\X_h}$ is the $L_2$ projection onto $\X_h$ as:
\begin{equation*}
	\int_{\Omega}\left(u - P_{\X_h}u\right)v_hdx = 0,\;\; \forall u\in L^2(\Omega),\, v_h\in X_h.
\end{equation*}
For any $\xi,\eta\in \X_h$, denote the $L_2$ inner product of $\xi,\eta$ on $\Omega$ by $\langle\xi,\eta\rangle$ with the $L^2$ norm $\Vert\xi\Vert =\langle\xi,\xi\rangle^{\frac12}$, and denote the $H^1$ norm of $\xi$ by $\Vert\xi\Vert_{H^1} =(\Vert\xi\Vert^2+\Vert\na\xi\Vert^2)^{\frac12}$.  We employ finite element scheme to solve the system \eqref{1.mu} and \eqref{1.ic} with the homogeneous Neumann boundary conditions.


Let $T > 0$ be the final time, and divide the time interval $[0, T]$ uniformly into $N$ subintervals with time step size $\tau = T/N$. The discrete time points are defined by
$
0 = t_0 < t_1 < \cdots < t_N = T, \quad \text{where } t_n := n\tau \text{ for } n = 0, 1, \ldots, N.
$
For a continuous function $v : [0, T] \rightarrow L^2(\Omega)$, we define the discrete value at time $t_n$ as
$
v^n := v(\cdot, t_n).
$
The backward difference quotient at time $t_n$ is given by
\[
\delta_t v^n := \frac{v^n - v^{n-1}}{\tau}.
\]
We now introduce the following stabilized numerical scheme:  
Given $(\phi^n, c^n) \in \X_h^2$ for $0 \leq n \leq N-1$, find $(\phi^{n+1}, c^{n+1}, \mu^{n+1}) \in {\X}_h^3$ such that the following system holds for all test functions $(\xi, \eta, \sigma) \in {\X}_h^3$:
\begin{align}
	\left\langle \frac{\phi^{n+1}-\phi^n}{\tau},\xi \right\rangle =& -\left\langle \na\mu^{n+1}-c^n\na\big(c^{n+1}-\phi^n\big),\na\xi\right\rangle,\label{2.phi}\\
	\left\langle \frac{c^{n+1}-c^n}{\tau},\eta\right\rangle  =& \left\langle c^n\na\mu^{n+1}-(c^n)^2\na\big(c^{n+1}-\phi^n\big),\na\eta\right\rangle -\left\langle \na\big(c^{n+1}-\phi^n\big),\na\eta\right\rangle,\label{2.c} \\
	\langle\mu^{n+1},\sigma\rangle =& \left\langle \na\phi^{n+1},\na\sigma\right\rangle +\frac1{\varepsilon^2}\left\langle f(\phi^{n})+S(\phi^{n+1}-\phi^n),\sigma\right\rangle  - \left\langle c^{n+1},\sigma\right\rangle.  \label{2.mu}
\end{align}
Here, $S$ is a stabilization parameter introduced to ensure the numerical stability of the scheme. The initial data $(\phi^0, c^0)$ are assumed to be sufficiently accurate and satisfy the following estimate:
\vspace{-0.4cm}
\begin{align}
	\left\| \phi^0 - \phi_0 \right\|_{H^1(\Omega)} + \left\| c^0 - c_0 \right\|_{H^1(\Omega)} \leq C(h^l + \tau), \label{er-in}
\end{align}
\vspace{-0.25cm}
for some constant $C > 0$ independent of $h$ and $\tau$.
\begin{lemma}[{\em \bf Discrete energy inequality}]\label{lem.Edelta}
Let $S > \frac{L}{2}$ for some constant $L > 0$. For each time step $n = 0, 1, \ldots, N-1$, suppose the numerical solution $(\phi^n, c^n) \in {\X}_h^2$ and the updated solution $(\phi^{n+1}, c^{n+1}, \mu^{n+1}) \in {\X}_h^3$ are given by the stabilized scheme \eqref{2.phi}--\eqref{2.mu}. Then, the following discrete energy inequality holds unconditionally:
	\begin{align*}
		\delta_t E^{n+1} 
		\leq& -\tau \left\| \nabla \mu^{n+1} - c^n \nabla (c^{n+1} - \phi^n) \right\|^2 
		- \tau \left\| \nabla (c^{n+1} - \phi^n) \right\|^2 
		- \frac{S}{\varepsilon^2} \left\| \delta_t \phi^{n+1} \right\|^2,
	\end{align*}
	where the discrete energy functional is defined by
	\[
	E^n := \frac{1}{2} \left\| \nabla \phi^n \right\|^2 + \frac{1}{\varepsilon^2} \left\langle F(\phi^n), 1 \right\rangle + \frac{1}{2} \left\| c^n \right\|^2 - \left\langle \phi^n, c^n \right\rangle.
	\]
	Moreover, the scheme preserves mass at each time step in the following sense:
	\[
	\left\langle \phi^{n+1}, 1 \right\rangle = \left\langle \phi^n, 1 \right\rangle, \quad 
	\left\langle c^{n+1}, 1 \right\rangle = \left\langle c^n, 1 \right\rangle.
	\]
\end{lemma}

\begin{proof} 
To establish the energy stability, we begin by selecting specific test functions in the scheme \eqref{2.phi}--\eqref{2.c}. Choose $\xi = \mu^{n+1} \in \X_h$ in \eqref{2.phi} and $\eta = c^{n+1} - \phi^n \in \X_h$ in \eqref{2.c}. This choice is motivated by the observation that the auxiliary term $h_c(\phi^n, c^{n+1}) = c^{n+1} - \phi^n$ is linear in both arguments, and hence belongs to the finite element space $\X_h$.
	
	Substituting these test functions into the respective equations and summing the resulting expressions, we obtain:
	\begin{align}
		&\frac{1}{\tau}\left\langle\delta_t\phi^{n+1},\mu^{n+1} \right\rangle
		+ \frac{1}{\tau}\left\langle\delta_t c^{n+1},c^{n+1}-\phi^n\right\rangle\label{2.aux} \\
		&= -\left\langle\na\mu^{n+1}-c^n\na(c^{n+1}-\phi^n)
		,\na\mu^{n+1} \right\rangle + \left\langle c^n\big(\na\mu^{n+1}-c^n
		\na(c^{n+1}-\phi^n)\big),\na(c^{n+1}-\phi^n)\right\rangle
		\nonumber \\
		&\phantom{xx}- \Vert\na(c^{n+1}-\phi^n)\Vert^2 = -\Vert\na\mu^{n+1}-c^n\na(c^{n+1}-\phi^n)\Vert^2
		- \Vert c^n\na(c^{n+1}-\phi^n)\Vert^2. \nonumber
	\end{align}
	Choosing the test function $\sigma=\frac1{\tau}\delta_t \phi^{n+1}$ in \eqref{2.mu} yields
	\begin{align}
		\frac1\tau\left\langle\mu^{n+1},\delta_t \phi^{n+1}\right\rangle = \frac1\tau\left\langle\na\phi^{n+1},\na\delta_t \phi^{n+1}\right\rangle +\frac1{\tau\varepsilon^2}\left\langle f(\phi^{n})+S\delta_t \phi^{n+1},\delta_t \phi^{n+1}\right\rangle  - \frac1\tau\left\langle c^{n+1},\delta_t \phi^{n+1}\right\rangle,\nonumber
	\end{align}
	which, together with the mass conservation for $c^n$, yields to
	\begin{align}
		&\frac{1}{\tau}\left\langle\mu^{n+1},\delta_t\phi^{n+1} \right\rangle
		+ \frac{1}{\tau}\left\langle \delta_t c^{n+1},c^{n+1}-\phi^n\right\rangle\label{2.aux2} \\
		&= \frac{1}{\tau}\left\langle\na\phi^{n+1},
		\na(\delta_t\phi^{n+1})\right\rangle + \frac{1}{\tau}\left\langle\frac{1}{\eps^2}
		f(\phi^{n})+\frac{S}{\eps^2}\delta_t \phi^{n+1}- c^{n+1},\delta_t\phi^{n+1}\right\rangle\noN\\&+ \frac{1}{\tau}\left\langle
		c^{n+1}-c^n,c^{n+1}-\phi^n\right\rangle \nonumber \\
		&\ge \frac{1}{2\tau}\big(\Vert\na\phi^{n+1}\Vert^2-\Vert\na\phi^n\Vert^2+\Vert\na(\delta_\phi^{n+1})\Vert^2\big)+ \frac{1}{\tau}\left\langle\frac1{\eps^2} F(\phi^{n+1}) + h(\phi^{n+1},c^{n+1}),1\right\rangle\nonumber\\
		&\phantom{xx}- \frac{1}{\tau}\left\langle\frac1{\eps^2} F(\phi^{n}) +h(\phi^{n},c^{n}),1\right\rangle
		+\frac{1}{\eps^2\tau}\left(S-\frac L2\right)\Vert \delta_t \phi^{n+1}\Vert ^2.\nonumber
	\end{align}
The energy stability follows directly from the combination of equations \eqref{2.aux} and \eqref{2.aux2}.
To verify mass conservation, we test equation \eqref{2.phi} with $\xi = 1$ and equation \eqref{2.c} with $\eta = 1$. This choice leads to the preservation of the discrete mass in both $\phi$ and $c$ variables, thereby completing the proof.
\end{proof}

In addition, the following fundamental estimates hold uniformly with respect to $h > 0$ and $\tau > 0$.
\begin{lemma}[{\em\bf Uniform a priori estimates}]\label{lemma-est}
	Let $\overline{\mu^{n+1}} := \frac{1}{|\Omega|} \int_\Omega \mu^{n+1}(x)\,dx = |\Omega|^{-1} \left\langle \mu^{n+1}, 1 \right\rangle$. Assume that the initial data $\phi_0, c_0 \in H^1(\Omega)$, and that assumptions \eqref{apf-1}--\eqref{apf-2} as well as the initial estimate \eqref{er-in} are satisfied. Then, there exists a constant $C > 0$ independent of $h$ and $\tau$, such that the following estimates hold:
	\begin{align}
		&\|\nabla \phi^{n+1}\|^2 + \|\phi^{n+1}\|^2 + \|c^{n+1}\|^2 + \sum_{k=0}^{n} \|\delta_t c^{k+1}\|^2 + \sum_{k=0}^{n} \|\nabla(\delta_t \phi^{k+1})\|^2 \nonumber\\
		&\quad + \sum_{k=0}^{n} \tau \|\nabla \mu^{k+1} - c^k \nabla(c^{k+1} - \phi^k)\|^2 + \sum_{k=0}^{n} \tau \|\nabla(c^{k+1} - \phi^k)\|^2 \leq C, \label{basic-1} \\
		&|\overline{\mu^{n+1}}| \leq C\left(\|f(\phi^n)\|^2 + 1\right) \leq C. \label{basic-2a}
	\end{align}
\end{lemma}
		
\begin{proof} Using \eqref{2.aux} and \eqref{2.aux2}, we have
	\begin{align}
		&\frac{1}{2}\big(\Vert\nabla \phi^{n+1}\Vert^2-\Vert\nabla \phi^n\Vert^2+\Vert\nabla (\delta_t\phi^{n+1})\Vert^2\big)+ \frac{1}{\eps^2}\langle F(\phi^{n+1}),1\rangle\noN\\& + \frac{1}{2}\Vert c^{n+1}\Vert^2- \langle\phi^{n+1},c^{n+1}\rangle- \frac{1}{\eps^2}\langle F(\phi^n),1\rangle\notag- \frac{1}{2}\Vert c^n\Vert^2+ \langle\phi^n,c^n\rangle + \frac{1}{2}\Vert \delta_t c^{n+1}\Vert^2\noN\\&+\tau\Vert\nabla \mu^{n+1}-c^n\nabla (c^{n+1}-\phi^n)\Vert^2 +\tau\Vert\nabla (c^{n+1}-\phi^n)\Vert^2\leq0.\nonumber
	\end{align}
By replacing the superscript with $k$ and summing over $k = 0$ to $n$, we obtain
	\begin{align}
		&\frac{1}{2}\Vert\nabla \phi^{n+1}\Vert^2+\frac{1}{2}\sum_{k=0}^n\Vert\nabla (\delta_t\phi^{k+1})\Vert^2+ \frac{1}{\eps^2}\langle F(\phi^{n+1}),1\rangle + \frac{1}{2}\Vert c^{n+1}\Vert^2- \langle\phi^{n+1},c^{n+1}\rangle\nonumber \\
		&+\sum_{k=0}^n\tau\Vert\nabla \mu^{k+1}-c^n\nabla (c^{k+1}-\phi^k)\Vert^2 +\sum_{k=0}^n\tau\Vert\nabla (c^{k+1}-\phi^k)\Vert^2+ \frac12\sum_{k=0}^{n}\Vert \delta_t c^{k+1}\Vert^2\notag\\
		&\leq \Vert\nabla \phi^0\Vert^2+ \frac{1}{\eps^2}\langle F(\phi^0),1\rangle + \Vert c^0\Vert^2- \langle\phi^0,c^0\rangle . \nonumber
	\end{align}
	Using assumption \eqref{apf-1} and the initial bounds, we establish the estimate \eqref{basic-1}.
	In addition, taking $\sigma = 1$ in \eqref{2.mu} and using the assumptions \eqref{apf-2}, we have
	\begin{align}
		&\vert\overline{\mu^{n+1}}\vert\leq C(\Vert f(\phi^{n})\Vert^2+1) \leq C(\Vert \phi^{n}\Vert^2+1),\nonumber
	\end{align} which, together with $\Vert \phi^n\Vert\le C$ in \eqref{basic-1}, gives the bound \eqref{basic-2a}.
\end{proof}
\vspace{-0.76cm}
\section{The existence of solution to the numerical scheme}\label{e-ns}
\vspace{-0.4cm}

The existence of a solution to the numerical scheme \eqref{2.phi}--\eqref{2.mu} is nontrivial due to the presence of cross-diffusion terms. Inspired by the approach in \cite[Section 4.4]{GKT22}, we construct a suitable mapping and invoke the Brouwer fixed-point theorem to guarantee the existence of solutions. The uniqueness of the solution can be established via energy-based arguments, following techniques similar to those presented in \cite{JuWa23}.
\begin{lemma}
	Assume that the parameters satisfy $K_1 + 2S > L + 2\varepsilon^2$ and $2S > L > 0$. Let $0 < \tau \leq 1$ and $0 < h \leq 1$ be such that $\frac{\tau}{h} \leq C$ for some constant $C > 0$. Given $(\phi^n, c^n) \in \X_h^2$, there exists a solution $(\phi^{n+1}, c^{n+1}, \mu^{n+1}) \in \X_h^3$ to the discrete system \eqref{2.phi}--\eqref{2.mu}.
\end{lemma}
\vspace{-0.35cm}
\begin{proof}
Following the approach of \cite{JuWa23}, we introduce the following inner product on the Hilbert space \(\X_h^3\):
\[
\langle(\phi_h, c_h, \mu_h), (\alpha_h, \beta_h, \gamma_h)\rangle 
:= \int_{\Omega} \left( \phi_h \alpha_h + c_h \beta_h + \mu_h \gamma_h \right) \, dx.
\]
Given the current time-step values \((\phi^n, c^n) \in \X_h^2\), we define the mapping 
$
Y : \X_h^3 \rightarrow \X_h^3
$
via the following relation:
	\begin{align*}
		\langle & Y(\phi_h,c_h,\mu_h),(\alpha_h,\beta_h,\gamma_h)\rangle= \int_{\Omega}\bigg(\frac{1}{\tau}(\phi_h-\phi^n)\alpha_h
		+ \big(\na\mu_h-c^n\na(c_h-\phi^n)\big)\cdot\na\alpha_h\bigg)dx \\
		&+ \int_{\Omega}\bigg\{\na\phi_h\cdot\na\gamma_h + \big(-\mu_h + \frac{1}{\eps^2}
		(f(\phi^n)+S(\phi_h-\phi^n))-c_h\big)\gamma_h\bigg\}dx \\
		&+ \int_{\Omega}\bigg(\frac{1}{\tau}(c_h-c^n)(\beta_h-\phi^n)
		- c^n\big(\na\mu_h-c^n\na(c_h-\phi^n)\big)\cdot
		\na(\beta_h-\phi^n)\bigg)dx
	\end{align*}
	for all test functions $(\alpha_h,\beta_h,\gamma_h)\in \X_h^3$. 
	
A zero of the mapping \(Y\) corresponds to a solution of the system \eqref{2.phi}--\eqref{2.mu}. To proceed via contradiction, we introduce a linear transformation 
$
g : \X_h^3 \rightarrow \X_h^3, \quad g(\phi_h, c_h, \mu_h) := \left( \mu_h, c_h, \frac{\phi_h}{\tau} - \frac{2\delta}{S^2} \mu_h \right),
$
where \(\delta \in (0, \tfrac{1}{2})\) is a small parameter to be determined later. The inverse of \(g\) is given by
\[
g^{-1}(\alpha_h, \beta_h, \gamma_h) := \left( \tau \left( \frac{2S^2}{\delta} \alpha_h + \gamma_h \right), \beta_h, \alpha_h \right).
\]
Furthermore, for a given radius \(R > 0\), define the closed ball
\[
B_R := \left\{ (\alpha_h, \beta_h, \gamma_h) \in \X_h^3 : \|(\alpha_h, \beta_h, \gamma_h)\|_2 \le R \right\},
\]
with the norm induced by the inner product:
$
\|(\alpha_h, \beta_h, \gamma_h)\|_2^2 := \langle (\alpha_h, \beta_h, \gamma_h), (\alpha_h, \beta_h, \gamma_h) \rangle.
$
	
Assume, for the sake of contradiction, that the continuous mapping \( Y \circ g^{-1} \) has no zeros in the closed ball \( B_R \subset \X_h^3 \). Following the strategy outlined in \cite[Section 4.4]{GKT22}, we apply Brouwer's fixed-point theorem to the auxiliary mapping
\[
G_R(\alpha_h, \beta_h, \gamma_h) := -R \frac{(Y \circ g^{-1})(\alpha_h, \beta_h, \gamma_h)}{\| (Y \circ g^{-1})(\alpha_h, \beta_h, \gamma_h) \|_2}, \quad \text{for } (\alpha_h, \beta_h, \gamma_h) \in B_R.
\]
This defines a continuous function from \( B_R \) into its boundary \( \partial B_R \), contradicting Brouwer’s fixed-point theorem. Therefore, a fixed point must exist in \( B_R \), i.e., there exists \((\alpha_h, \beta_h, \gamma_h) \in B_R\) such that
\[
G_R(\alpha_h, \beta_h, \gamma_h) = (\alpha_h, \beta_h, \gamma_h), \quad \text{with } \|G_R(\alpha_h, \beta_h, \gamma_h)\|_2 = \|(\alpha_h, \beta_h, \gamma_h)\|_2 = R.
\]

By the definition of \( g \), there exists some \( (\phi_h, c_h, \mu_h) \in \X_h^3 \) such that \( g(\phi_h, c_h, \mu_h) = (\alpha_h, \beta_h, \gamma_h) \). From the definition of \( g \), we have
$
\alpha_h = \mu_h, \quad \beta_h = c_h, \quad \gamma_h = \frac{\phi_h}{\tau} - \frac{2\delta}{S^2}\mu_h.
$
Substituting into the inner product, we write
\[
\langle Y(\phi_h, c_h, \mu_h), (\alpha_h, \beta_h, \gamma_h) \rangle = I_1 + I_2 + I_3,
\]
where the terms are grouped as follows:
\begin{align*}
	I_1 &= \int_{\Omega} \left( \left| \nabla \mu_h - c^n \nabla(c_h - \phi^n) \right|^2 + \frac{1}{\tau} |\nabla \phi_h|^2 + |\nabla(c_h - \phi^n)|^2 - \frac{2\delta}{S^2} \nabla \phi_h \cdot \nabla \mu_h \right) dx, \\
	I_2 &= \int_{\Omega} \left( -\frac{1}{\tau} \phi^n \mu_h - \frac{1}{\tau} c_h \phi_h + \frac{2\delta}{S^2}(\mu_h + c_h)\mu_h + \frac{1}{\tau}(c_h - c^n)(c_h - \phi^n) \right) dx, \\
	I_3 &= \frac{1}{\eps^2} \int_{\Omega} \left( f(\phi^n) + S(\phi_h - \phi^n) \right)\left( \frac{\phi_h}{\tau} - \frac{2\delta}{S^2} \mu_h \right) dx.
\end{align*}
As shown in \cite{JuWa23}, the terms \( I_1, I_2, I_3 \) can be estimated from below:
\begin{align*}
	I_1 &\ge \frac{1}{2\tau} \int_{\Omega} |\nabla \phi_h|^2 dx 
	+ \int_{\Omega} \left| \nabla \mu_h - c^n \nabla(c_h - \phi^n) \right|^2 dx
	+ \int_{\Omega} |\nabla(c_h - \phi^n)|^2 dx \noN\\&
	- 2\tau\delta^2 S^{-4} Ch^{-1} \int_{\Omega} |\mu_h|^2 dx, \\
	I_2 &\ge \int_{\Omega} \left( \frac{c_h^2}{4\tau} + \frac{\delta}{S^2} \left( \frac{3}{2} - \frac{4\delta \tau}{S^2} \right) \mu_h^2 \right) dx 
	- \frac{1}{\tau} \int_{\Omega} \phi_h^2 dx 
	- \frac{C}{\tau} \int_{\Omega} \left( \left( 1 + \frac{S^2}{\delta \tau} \right) |\phi^n|^2 + |c^n|^2 + 1 \right) dx, \\
	I_3 &\ge \frac{1}{2\eps^2 \tau} \left( K_1 + 2S - L - \frac{4\delta \tau}{\eps^2} \right) \int_{\Omega} \phi_h^2 dx 
	- \frac{\delta}{2S^2} \int_{\Omega} \mu_h^2 dx - C^n.
\end{align*}
Summing these estimates, we obtain
\begin{align*}
	\langle Y(\phi_h, c_h, \mu_h), (\alpha_h, \beta_h, \gamma_h) \rangle 
	&\ge \int_{\Omega} \bigg\{ \frac{1}{2\tau} |\nabla \phi_h|^2 
	+ \left| \nabla \mu_h - c^n \nabla(c_h - \phi^n) \right|^2 
	+ |\nabla(c_h - \phi^n)|^2 \\
	&\qquad + \frac{1}{2\eps^2\tau} \left( K_1 + 2S - L - 2\eps^2 - \frac{4\delta \tau}{\eps^2} \right) \phi_h^2 
	+ \frac{c_h^2}{4\tau} \\
	&\qquad + \frac{\delta}{S^2} \left( 1 - \frac{4\delta \tau}{S^2} - \frac{2C\delta \tau}{S^2 h} \right) \mu_h^2 \bigg\} dx - C^n.
\end{align*}
Choose \( \delta > 0 \) sufficiently small so that
\[
K_1 + 2S - L - 2\eps^2 - \frac{4\delta \tau}{\eps^2} > 0 
\quad \text{and} \quad 
1 - \frac{4\delta \tau}{S^2} - \frac{2C\delta \tau}{S^2 h} > 0,
\]
which is possible under the assumptions \( K_1 + 2S > L + 2\eps^2 \), \( S > 0 \), and \( \frac{\tau}{h} \le C \). Then, we obtain
\[
\langle Y(\phi_h, c_h, \mu_h), (\alpha_h, \beta_h, \gamma_h) \rangle 
\ge C_\delta \| (\phi_h, c_h, \mu_h) \|_2^2 - C^n,
\]
for some constant \( C_\delta > 0 \).

By equivalence of norms in finite-dimensional spaces, there exists \( C'_\delta > 0 \) such that
\[
\| (\phi_h, c_h, \mu_h) \|_2 \ge C'_\delta \| g(\phi_h, c_h, \mu_h) \|_2 
= C'_\delta \| (\alpha_h, \beta_h, \gamma_h) \|_2 = C'_\delta R.
\]
Thus, if \( R > 0 \) is sufficiently large, then
\begin{equation} \label{3.pos}
	\langle Y(\phi_h, c_h, \mu_h), (\mu_h, c_h, \phi_h/\tau - 2\delta S^{-2} \mu_h) \rangle 
	= \langle Y(\phi_h, c_h, \mu_h), (\alpha_h, \beta_h, \gamma_h) \rangle 
	> 0.
\end{equation}
On the other hand, since \( (\alpha_h, \beta_h, \gamma_h) = g(\phi_h, c_h, \mu_h) \in B_R \) is a fixed point of \( G_R \), we have
\[
G_R(g(\phi_h, c_h, \mu_h)) = -R \frac{Y(\phi_h, c_h, \mu_h)}{\| Y(\phi_h, c_h, \mu_h) \|_2},
\]
and hence,
\begin{align*}
	\langle Y(\phi_h, c_h, \mu_h), (\mu_h, c_h, \phi_h/\tau - 2\delta S^{-2} \mu_h) \rangle 
	&= -\frac{1}{R} \| Y(\phi_h, c_h, \mu_h) \|_2 
	\left\| (\mu_h, c_h, \phi_h/\tau - 2\delta S^{-2} \mu_h) \right\|_2^2\noN\\& \le 0,
\end{align*}
which contradicts \eqref{3.pos}. This contradiction implies that \( Y \circ g^{-1} \) must attain a zero in \( B_R \) for sufficiently large \( R > 0 \). Hence, the scheme \eqref{2.phi}--\eqref{2.mu} admits a solution.
\end{proof}
\vspace{-0.65cm}

\section{Error estimate}\label{sec-er}
\vspace{-0.25cm}

This section is devoted to deriving an error estimate for the fully discrete scheme \eqref{2.phi}--\eqref{2.mu}. Our approach builds upon the framework developed in \cite{CaSh18}, which provided error estimates for a fully discrete scheme associated with a Cahn--Hilliard phase-field model for two-phase incompressible flows. A major challenge in the current analysis stems from the presence of cross-diffusion terms in the system. These terms inhibit the derivation of a uniform $L^2$ bound for $\nabla \mu^{n+1}$ with respect to $n$, $h > 0$, and $\tau > 0$, as well as bounds on the inverse Laplacian of the error, in contrast to the standard Cahn--Hilliard model.

We consider spatial dimensions $d = 1, 2, 3$, and assume the following regularity conditions on the weak solutions of \eqref{1.mu}:
\begin{align}
	(\phi,c) &\in (L^\infty([0,T]; H^{1+l}(\Omega)))^2, \quad \partial_t(\phi,c) \in (L^\infty([0,T]; H^{1+l}(\Omega)))^2, \nonumber\\
	\partial_{tt}(\phi,c) &\in (L^\infty([0,T]; L^2(\Omega)))^2, \quad \mu \in L^\infty([0,T]; H^{1+l}(\Omega)). \label{re-reg}
\end{align}
Here, $l \geq 1$ when $d = 1$, and $l > d/2$ when $d = 2$ or $3$.

Furthermore, for a sequence of discrete functions $u^k \in K_s$, defined for $k = 0, 1, \ldots, N$ with $K_s$ denoting a suitable Sobolev space, we introduce the following discrete norms that will be used throughout the analysis:
\begin{equation*}
	\Vert u^k \Vert_{L^\infty(K_s)} = \max_{0 \leq k \leq N} \Vert u^k \Vert_{K_s}, \quad 
	\Vert u^k \Vert_{L^p(K_s)} = \left( \tau \sum_{k=0}^N \Vert u^k \Vert_{K_s}^p \right)^{1/p}, \quad p > 1.
\end{equation*}
Denote the weak solutions of \eqref{1.mu} by $\phi(t),\,c(t),\,\mu(t)$. We define $\phi_h(t),\,c_h(t),\,\mu_h(t)$ by the following projection
\begin{equation}\label{1.pro}
	\left\{
	\begin{array}{lr}
		\begin{aligned}
			& \langle\nabla \phi_h,\nabla \xi\rangle = \langle\nabla \phi,\nabla \xi\rangle,\,\forall\xi\in \X_h,\,\langle\phi_h-\phi,1\rangle=0;\\
			& \langle\nabla c_h,\nabla \eta\rangle = \langle\nabla c,\nabla \eta\rangle,\,\forall\eta\in \X_h,\,\langle c_h-c,1\rangle=0;\\
			& \langle\nabla \mu_h,\nabla \sigma\rangle = \langle\nabla \mu,\nabla \sigma\rangle,\,\forall\sigma\in \X_h,\,\langle\mu_h-\mu,1\rangle=0.\\
		\end{aligned}
	\end{array}\right.
\end{equation}
According to the property of the projection defined by \eqref{1.pro}, the following inequality is satisfied for some $C>0$:
\begin{align}
	&\Vert\phi-\phi_h\Vert_{L^\infty([0,T];L^2(\Omega))} + h\Vert\phi-\phi_h\Vert_{L^\infty([0,T];H^1(\Omega))}\leq C h^{1+l}\Vert \phi\Vert_{L^\infty([0,T];H^{l+1}(\Omega))},\label{pr-reg-1}\\
	&\Vert c-c_h \Vert_{L^\infty([0,T];L^2(\Omega))}+h\Vert c-c_h \Vert_{L^\infty([0,T];H^1(\Omega))}\leq C h^{1+l}\Vert c\Vert_{L^\infty([0,T];H^{l+1}(\Omega))},\label{pr-reg-2}\\
	&\Vert\mu-\mu_h\Vert_{L^\infty([0,T];L^2(\Omega))}+h\Vert\mu-\mu_h\Vert_{L^\infty([0,T];H^1(\Omega))}\leq C h^{1+l}\Vert \mu\Vert_{L^\infty([0,T];H^{l+1}(\Omega))}.\label{pr-reg-3}
\end{align}
In addition, we introduce $L_h\xi$, the inverse Laplacian operator of $\xi\in \X_h$ in $\Omega$, by
\begin{equation}\label{def.iap}
	\left\{
	\begin{array}{lr}
		\begin{aligned}
			& 	\left\langle \nabla L_{h}\xi,\nabla  \eta\right\rangle = \left\langle \xi,\eta\right\rangle +\frac{1}{\vert\Omega\vert}\left\langle \xi,1\right\rangle\left\langle \eta,1\right\rangle,\, \forall\eta\in \X_h,\,\langle L_{h}\xi,1\rangle=0.
		\end{aligned}
	\end{array}\right.
\end{equation}
Now, we define the error functions
\begin{align} e_\phi^n= \phi^n - \phi_h(t_{n}),\,e_c^n= c^n - c_h(t_{n}),e_\mu^n= \mu^n - \mu_h(t_{n}).\label{err-fun}\end{align}
The main results of this section are as follows.
\begin{thm}\label{thm.od}
	Let $(\phi^n,c^n,\mu^n)$ be the solution to the numerical scheme \eqref{2.phi}-\eqref{2.mu} of the problem \eqref{1.mu}-\eqref{1.ic}. Let $S>0$, let $e_\phi^n, e_c^n, e_\mu^n$ be given by \eqref{err-fun}. Under assumptions \eqref{apf-1},\eqref{apf-2},\eqref{f-trc}, \eqref{er-in} and \eqref{re-reg},
	if the time step $\tau>0$ of discretization is small enough, the following error estimates can be established:
	\begin{align}
		\Vert e_\phi^n\Vert_{L^\infty(H^1(\Omega))}+\Vert e_c^n\Vert_{L^\infty(L^2(\Omega))}+\Vert \nabla  e_c^n\Vert_{L^2(L^2(\Omega))}\lesssim(\tau+h^{l}),\Vert \nabla e_\mu^n\Vert_{L^\frac43(L^\frac65(\Omega))}\lesssim\left(\tau+h^{l}\right)^\frac34 \label{est.main-S}
	\end{align}
	if $h^l\le\tau$, for some positive constant $C>0$ independent of $\tau>0$ and $h>0$.
	
	 Moreover, we have the optimal estimates:
	\begin{align}
		\Vert \phi^n - \phi(t_{n})\Vert_{L^\infty(H^1(\Omega))}+\Vert c^n - c(t_{n})\Vert_{L^\infty(L^2(\Omega))}+\Vert \nabla (c^n - c(t_{n}))\Vert_{L^2(L^2(\Omega))}\lesssim \tau+h^{l}.\label{est.main-SS}
	\end{align}
\end{thm}
To complete the proof of Theorem \ref{thm.od}, we derive the equations satisfied by the error functions and the associated remainder terms. The error estimate is then established by employing a discrete Gr\"onwall inequality. The primary challenge in this analysis arises from the nonlinear cross-diffusion terms in the system. These terms prevent us from obtaining a uniform estimate of the form 
$
\tau \sum_{k=0}^n \|\nabla e_{\mu}^{k+1}\|^2,
$
for the error function $e_{\mu}^{n+1}$, uniformly with respect to $h > 0$ and $\tau > 0$. 

Nevertheless, this obstacle can be circumvented by deriving alternative bounds. Specifically, we establish uniform-in-$h$ and uniform-in-$\tau$ estimates for the following quantities:
\[
\tau \sum_{k=0}^n \left\|\nabla e_{\mu}^{k+1} - c^k \nabla(e_c^{k+1} - e_\phi^k)\right\|^2 
\quad \text{and} \quad 
\tau \left\|c^n \nabla(e_c^{n+1} - e_\phi^n)\right\|_{L^{\frac{6}{5}}}^{\frac{4}{3}},
\]
as shown in \eqref{basic-3-1} below.

We now derive the equations for the error functions. By applying a Taylor expansion and utilizing the identity \eqref{1.pro}, the continuous system \eqref{1.mu}--\eqref{1.ic} can be reformulated at time $t_{n+1}$ as follows:
\begin{align}\label{e1}
	\left\langle \frac{\phi_h(t_{n+1}) - \phi_h(t_n)}{\tau}, \xi \right\rangle 
&	= -\left\langle \nabla \mu_h(t_{n+1}) - c_h(t_n) \nabla \big(c_h(t_{n+1}) - \phi_h(t_n)\big), \nabla \xi \right\rangle \nonumber\\&
	+ \tilde{R}_\phi^{n+1}(\xi, \nabla \xi), \\
	\left\langle \frac{c_h(t_{n+1}) - c_h(t_n)}{\tau}, \eta \right\rangle  
	&= \left\langle c_h(t_n) \nabla \mu_h(t_{n+1}) - \big(c_h^2(t_n) + 1\big) \nabla \big(c_h(t_{n+1}) - \phi_h(t_n)\big), \nabla \eta \right\rangle \nonumber\\&
	+ \tilde{R}_c^{n+1}(\eta, \nabla \eta), \\
	\left\langle \mu_h(t_{n+1}), \sigma \right\rangle 
	&= \left\langle \nabla \phi_h(t_{n+1}), \nabla \sigma \right\rangle 
	+ \frac{1}{\varepsilon^2} \left\langle f\big(\phi_h(t_n)\big) + S \big(\phi_h(t_{n+1}) - \phi_h(t_n)\big), \sigma \right\rangle\nonumber \\
	&\quad - \left\langle c_h(t_{n+1}), \sigma \right\rangle 
	+ \tilde{R}_\mu^{n+1}(\sigma).
\end{align}
Here, the remainder terms $\tilde{R}_\phi^{n+1}(\xi, \nabla \xi)$, $\tilde{R}_c^{n+1}(\eta, \nabla \eta)$, and $\tilde{R}_\mu^{n+1}(\sigma)$ are defined by:
\begin{align}
	\tilde{R}_\phi^{n+1}(\xi, \nabla \xi) 
	&= \langle \frac{\phi_h(t_{n+1}) - \phi_h(t_n)}{\tau} - \phi_t(t_{n+1}), \xi\rangle \nonumber\\
	&\quad - \Big\langle c(t_{n+1}) \nabla \big(c(t_{n+1}) - \phi(t_{n+1})\big) 
	- c_h(t_n) \nabla \big(c_h(t_{n+1}) - \phi_h(t_n)\big), \nabla \xi \Big\rangle, \label{rem.til.phi} \\
	\tilde{R}_c^{n+1}(\eta, \nabla \eta) 
	&= \left\langle \frac{c_h(t_{n+1}) - c_h(t_n)}{\tau} - c_t(t_{n+1}), \eta \right\rangle \nonumber \\
	&\quad + \Big\langle c(t_{n+1}) \nabla \mu(t_{n+1}) - c^2(t_{n+1}) \nabla \big(c(t_{n+1}) - \phi(t_{n+1})\big) \nonumber\\
	&\qquad - \big(c_h(t_n) \nabla \mu_h(t_{n+1}) + c_h^2(t_n) \nabla \big(c_h(t_{n+1}) - \phi_h(t_n)\big)\big), \nabla \eta \Big\rangle \nonumber \\
	&\quad + \left\langle \nabla \phi(t_{n+1}) - \nabla \phi_h(t_n), \nabla \eta \right\rangle, \label{rem.til.c} \\
	\tilde{R}_\mu^{n+1}(\sigma) 
	&= \frac{1}{\varepsilon^2} \left\langle f\big(\phi(t_{n+1})\big) - f\big(\phi_h(t_n)\big) - S \big(\phi_h(t_{n+1}) - \phi_h(t_n)\big), \sigma \right\rangle \nonumber \\
	&\quad - \left\langle \mu(t_{n+1}) - \mu_h(t_{n+1}) + c(t_{n+1}) - c_h(t_{n+1}), \sigma \right\rangle. \label{rem.til.mu}
\end{align}
Moreover, we introduce the notation for backward differences: 
\[
\delta_t e_\phi^{n+1} := e_\phi^{n+1} - e_\phi^n, \quad \delta_t e_c^{n+1} := e_c^{n+1} - e_c^n.
\]
By subtracting the interpolation equations \eqref{e1} from the discrete numerical scheme \eqref{2.phi}--\eqref{2.c}, we obtain the following error equations:
\begin{align}
	&\frac{1}{\tau}\left\langle \delta_t e_\phi^{n+1},\xi \right\rangle = -\left\langle \nabla e_\mu^{n+1}-c^n\nabla (e_c^{n+1}-e_\phi^{n}),\nabla \xi\right\rangle -\tilde R^{n+1}_\phi\big(\xi,\nabla \xi\big)+R^{n+1}_\phi(\nabla \xi),\label{e1.phi}\\
	&\frac{1}{\tau}\left\langle \delta_t e_c^{n+1},\eta\right\rangle  = \left\langle c^n\nabla e_\mu^{n+1}-(c^{n2}+1)\nabla (e_c^{n+1}-e_\phi^{n}),\nabla \eta\right\rangle-\tilde R^{n+1}_c(\eta,\nabla \eta) +R^{n+1}_c(\nabla \eta),\label{e1.c}\\
	&\langle e_\mu^{n+1},\sigma\rangle = \left\langle \nabla e_\phi^{n+1},\nabla \sigma\right\rangle  - \left\langle e_c^{n+1},\sigma\right\rangle  + \frac S{\varepsilon^2}(e_\phi^{n+1}-e_\phi^n,\sigma) -\tilde R^{n+1}_\mu(\sigma) +R^{n+1}_\mu(\sigma),\label{e1.mu}
\end{align}
where $R^{n+1}_\phi(\nabla \xi),\,R^{n+1}_c(\nabla \eta)$ and $R^{n+1}_\mu(\sigma)$ are also the remainders defined as follows:
\begin{align}
	&R^{n+1}_\phi(\nabla \xi) 
	=- \left\langle -e_c^{n}\nabla \big(c_h(t_{n+1})-\phi_h(t_n)\big),\nabla \xi\right\rangle,\label{rem.phi}\\
	&R^{n+1}_c(\nabla \eta)
	=\big\langle e_c^{n}\nabla \mu_h(t_{n+1})-e_c^n(c^n+c_h(t_n))\nabla \big(c_h(t_{n+1})-\phi_h(t_n)\big),\nabla \eta\big\rangle\label{rem.c}\\
	&R^{n+1}_\mu(\sigma)= \frac1{\varepsilon^2}\left\langle f(\phi^n)-f(\phi_h(t_n)),\sigma\right\rangle.\label{rem.mu}
\end{align}

\begin{lemma}\label{lem.avg}
	Let $\overline{\delta_t e_\phi^{n+1}} = \vert\Omega\vert^{-1}\langle\delta e_\phi^{n+1},1\rangle $ and $\overline{ e_\mu^{n+1}} =  \vert\Omega\vert^{-1}\langle e_\mu^{n+1},1\rangle$. We have
	\begin{equation}
		\vert\overline{\delta_t e_\phi^{n+1}}\vert\leq C\tau(\tau+ h^{1+l}),\quad\vert\overline{e_\phi^{n+1}}\vert+\vert\overline{e_c^{n+1}}\vert\leq C(\tau+ h^{l}),\quad\vert\overline{ e_\mu^{n+1}}\vert\leq C(\tau+ h^{l} +\Vert e_\phi^n\Vert).\label{av-1}
	\end{equation}
\end{lemma}
\begin{proof}
	Choosing $\xi =1 $ in \eqref{e1.phi}, we have
	\begin{align}
		&\vert\overline{\delta_t e_\phi^{n+1}}\vert = \frac1{\vert\Omega\vert}\vert(\delta_t e_\phi^{n+1},1) \vert=\frac{\tau}{\vert\Omega\vert}\vert(-\tilde R^{n+1}_\phi(1,0)+R^{n+1}_\phi(0))\vert\nonumber\\&=\frac{\tau}{\vert\Omega\vert}\left\vert\left\langle \frac{\phi_h(t_{n+1})-\phi_h(t_{n})}{\tau} - \phi_t(t_{n+1}),1\right\rangle\right\vert\nonumber\\
		&=\frac{\tau}{\vert\Omega\vert}\left\langle \frac{\phi_h(t_{n+1})-\phi_h(t_{n})-\delta_t \phi(t_{n+1})}{\tau} + \frac{\delta_t\phi(t_{n+1})}{\tau} - \phi_t(t_{n+1}),1\right\rangle.\nonumber
	\end{align}
	Using \eqref{pr-reg-1}, taking the partial derivative of \eqref{1.pro} with respect to $t$, we have
	\begin{align}
		&\vert\overline{\delta_t e_\phi^{n+1}}\vert=\frac{\tau}{\vert\Omega\vert}\left\vert\left\langle \frac{\phi_h(t_{n+1})-\phi_h(t_{n})}{\tau} - \phi_t(t_{n+1}),1\right\rangle\right\vert \leq C\tau(\tau+ h^{1+l}).\label{phi.av}
	\end{align} And, hence, we have \begin{align}\vert\overline{e_\phi^{n+1}}\vert\leq C(\tau+ h^{l}),\, \vert\overline{e_c^{n+1}}\vert\leq C(\tau+ h^{l}).\label{ephi-ec}\end{align}
	Also, similarly, taking $\sigma = 1$ in \eqref{e1.mu}, we have
	\begin{align}
		\vert \overline{ e_\mu^{n+1}}\vert  &= \frac1{\vert\Omega\vert}\vert(e_\mu^{n+1},1)\vert \leq \frac1{\vert\Omega\vert}\vert (e_c^{n+1},1)-\tilde R^{n+1}_\mu(1,0) +R^{n+1}_\mu(1)\vert +\frac S{\varepsilon^2}\vert \overline{\delta_t e_\phi^{n+1}}\vert \label{mu.av}\\
		&\leq\frac{\tau}{\vert\Omega\vert}\sum_{k=0}^{n}\vert  (-\tilde R^{k+1}_c(1,0) + R^{k+1}_c(0))\vert +\frac{1}{\vert\Omega\vert}\vert (\tilde R^{n+1}_\mu(1)-R^{n+1}_\mu(1))\vert +\frac S{\varepsilon^2}\vert \overline{\delta_t e_\phi^{n+1}}\vert \notag\\
		&\leq C(\tau+ h^{l} +\Vert e_\phi^n\Vert)\notag.
	\end{align} Combining \eqref{phi.av}, \eqref{ephi-ec} and \eqref{mu.av}, we get \eqref{av-1}.
\end{proof}

Next, choosing $\xi=e_\mu^{n+1}, e_\phi^{n+1}, L_h(\delta_t e_\phi^{n+1})$ in \eqref{e1.phi}, we have
\begin{align}
	\frac1{\tau}\left\langle \delta_t e_\phi^{n+1},e_\mu^{n+1} \right\rangle =&-\left\langle \nabla e_\mu^{n+1}-c^n\nabla (e_c^{n+1}-e_\phi^{n}),\nabla e_\mu^{n+1}\right\rangle -\tilde R_\phi^{n+1}(e_\mu^{n+1},\nabla e_\mu^{n+1})+R^{n+1}_\phi(e_\mu^{n+1}),\label{mu.n1}\\
	\frac1{\tau}\left\langle\delta_t e_\phi^{n+1},e_\phi^{n+1} \right\rangle =&-\left\langle\nabla e_\mu^{n+1}-c^n\nabla (e_c^{n+1}-e_\phi^{n}),\nabla e_\phi^{n+1}\right\rangle-\tilde R^{n+1}_\phi(e_\phi^{n+1},\nabla e_\phi^{n+1})+R^{n+1}_\phi(e_\phi^{n+1})\label{phi.n0}
\end{align}
and
\begin{align}
	\frac1{\tau}\left\langle \delta_t e_\phi^{n+1},L_h(\delta_t e_\phi^{n+1}) \right\rangle 
	&=-\left\langle \nabla e_\mu^{n+1}-c^n\nabla (e_c^{n+1}-e_\phi^{n}),\nabla L_h(\delta_t e_\phi^{n+1})\right\rangle\label{phi.n-1}\\
	&\phantom{xx} -\tilde R_\phi(L_h(\delta_t e_\phi^{n+1}),\nabla L_h(\delta_t e_\phi^{n+1}))+R_\phi(L_h(\delta_t e_\phi^{n+1})).\notag
\end{align}

Similarly, taking $\eta = e_c^{n+1} -e_\phi^n$, and $\sigma = \delta_t e_\phi^n$, we have
\begin{align}
	\frac1{\tau}\left\langle \delta_t e_c^{n+1},e_c^{n+1}-e_\phi^n\right\rangle  =& \left\langle c^n\nabla e_\mu^{n+1}-(c^n)^2\nabla (e_c^{n+1}-e_\phi^{n}),\nabla (e_c^{n+1}-e_\phi^n)\right\rangle -\Vert\nabla (e_c^{n+1}-e_\phi^{n})\Vert^2\notag\\
	& -\tilde R_c^{n+1}(e_c^{n+1}-e_\phi^n, \nabla (e_c^{n+1}-e_\phi^n)) +R_c^{n+1}(e_c^{n+1}-e_\phi^n),\label{c.n0}\\
	\left\langle \nabla e_\phi^{n+1},\nabla (\delta_t e_\phi^{n+1})\right\rangle =& \langle e_\mu^{n+1},\delta_t e_\phi^{n+1}\rangle + \left\langle e_c^{n+1},\delta_t e_\phi^{n+1}\right\rangle - \frac S{\varepsilon^2}\Vert\delta_t e_\phi^{n+1}\Vert^2\notag\\
	& +\tilde R_\mu^{n+1}(\delta_t e_\phi^{n+1}) -R_\mu^{n+1}(\delta_t e_\phi^{n+1}),\label{phi.n1}
\end{align}

An inequality of the error functions can be established now.

 By taking $\tau\cdot$\eqref{mu.n1}$+\tau\cdot$
\eqref{phi.n0}$+$\eqref{phi.n-1}$+\tau\cdot$\eqref{c.n0}$+$\eqref{phi.n1} and canceling some terms concerning $\left\langle\delta_t e_\phi^{n+1},e_\mu^{n+1} \right\rangle$, with the help of \eqref{def.iap}, we have
\begin{align}\label{3.est.f}
	&\frac{1}2(\Vert e_\phi^{n+1}\Vert^2-\Vert e_\phi^n\Vert^2+\Vert\delta_t e_\phi^{n+1}\Vert^2)+\frac1{2\tau}\Vert \nabla L_h(\delta_t e_\phi^{n+1})\Vert^2+\frac{1}{2}(\Vert e_c^{n+1}\Vert^2-\Vert e_c^n\Vert^2+\Vert\delta_t e_c^{n+1}\Vert^2)\notag\\
	&+\left\langle e_c^n,e_\phi^n\right\rangle+\frac{1}2(\Vert \nabla  e_\phi^{n+1}\Vert^2-\Vert \nabla e_\phi^n\Vert^2+\Vert\nabla (\delta_t e_\phi^{n+1})\Vert^2)+\tau\Vert\nabla (e_c^{n+1}-e_\phi^{n})\Vert^2\\
	&\leq -\frac{\tau}4\Vert \nabla e_\mu^{n+1}- c^n\nabla (e_c^{n+1}-e_\phi^{n})\Vert^2+C\tau\|\nabla e_\phi^{n+1}\|^2-\frac S{\varepsilon^2}\Vert\delta_t e_\phi^{n+1}\Vert^2+\left\langle e_c^{n+1},e_\phi^{n+1}\right\rangle  \notag\\
	&\phantom{xx}{}-\tau\tilde R^{n+1}_\phi(e_\phi^{n+1},\nabla e_\phi^{n+1})+\tau R^{n+1}_\phi(e_\phi^{n+1})-{\tau}\tilde R^{n+1}_\phi(e_\mu^{n+1},\nabla e_\mu^{n+1})+{\tau} R^{n+1}_\phi(\nabla e_\mu^{n+1})\nonumber\\
	&\phantom{xx}{}-\tilde R_\phi(L_h(\delta_t e_\phi^{n+1}),\nabla L_h(\delta_t e_\phi^{n+1}))+R_\phi(L_h(\delta_t e_\phi^{n+1}))\notag\\
	&\phantom{xx}{}-\tau\tilde R^{n+1}_c(e_c^{n+1}-e_\phi^n,\nabla (e_c^{n+1}-e_\phi^n)) +\tau R^{n+1}_c(\nabla(e_c^{n+1}-e_\phi^n))+\tilde R^{n+1}_\mu(\delta_t e_\phi^{n+1}) -R^{n+1}_\mu(\delta_t e_\phi^{n+1}).\nonumber
\end{align}
To estimate the remainder $-\tau\tilde R^{n+1}_\phi(e_\mu^{n+1},\nabla  e_\mu^{n+1})$ on the right-hand side of \eqref{3.est.f}, we need establish the following Lemma because we can not get the uniform bound for $\tau\|\nabla e_{\mu}^{n+1}\|^2$ in $h>0$ and $\tau>0$ due to the nonlinear cross diffusion of the error system.
\begin{lemma}\label{lemma-est-1} Assume that $\phi_0,c_0\in H^1(\Omega)$ and let the assumptions \eqref{apf-1}-\eqref{apf-2} and the initial bound \eqref{er-in} hold. Then there exists some positive constant $C$ such that \begin{align}
		&\kappa\tau\Vert c^n\nabla(e_c^{n+1}-e_\phi^{n})\Vert^\frac43_{L^\frac65}\le C\kappa^3\tau \Vert c^n\Vert_{H^1}^2+\tau\Vert\nabla(e_c^{n+1}-e_\phi^{n})\Vert^2\label{basic-3-1}\end{align} for any positive constant $\kappa$.
\end{lemma}
\begin{proof}
	Using the Holder's inequality, the Young's inequality and the interpolation inequality, we have
	\begin{align}
		\kappa\tau&\Vert c^n\nabla(e_c^{n+1}-e_\phi^{n})\Vert^\frac43_{L^{\frac65}}\leq \kappa\tau\Vert c^n\Vert_{L^3}^\frac43\Vert\nabla(e_c^{n+1}-e_\phi^{n})\Vert^\frac43\leq  \tau\Vert\nabla(e_c^{n+1}-e_\phi^{n})\Vert^2+\kappa^3\tau\Vert c^n\Vert_{L^3}^4\nonumber\\
		&\leq  \tau\Vert\nabla(e_c^{n+1}-e_\phi^{n})\Vert^2+\kappa^3\tau\Vert c^n\Vert_{L^2}^2\Vert c^n\Vert_{H^1}^2,\nonumber
	\end{align}
	which, together with bounds that $\Vert c^n\Vert^2\le C$ according to $\eqref{basic-1}$, gives the estimate \eqref{basic-3-1}.
\end{proof}

Next, we establish bounds for the remainders on the right-hand side of \eqref{3.est.f}.
\begin{lemma}\label{lem.est.til}
	Under assumptions \eqref{apf-1}, \eqref{apf-2}, \eqref{er-in}, \eqref{re-reg} and \eqref{pr-reg-1}-\eqref{pr-reg-3}, we have
	\begin{align*}
		&-\tau\tilde R^{n+1}_\phi(e_\phi^{n+1},\nabla  e_\phi^{n+1})\leq C\tau \big(h^{2l}+\tau^2\big)+\tau\kappa\Vert e_\phi^{n+1}\Vert_{H^1}^2,\\
		&-\tau\tilde R^{n+1}_\phi(e_\mu^{n+1},\nabla  e_\mu^{n+1})\leq C\tau\big(h^{2l}+\tau^2+\Vert e_\phi^n\Vert^2\big)+\kappa\tau\Vert \nabla e_\mu^{n+1}-c^n\nabla(e_c^{n+1}-e_\phi^{n})\Vert^2\\
		&\quad\quad +\kappa\tau\Vert\nabla(e_c^{n+1}-e_\phi^{n})\Vert^2+C\kappa (h^{2l}+\tau^2)\tau\|c^n\|_{H^1}^2,\\
		&-\tilde R_\phi(L_h(\delta e_\phi^{n+1}),\nabla L_h(\delta e_\phi^{n+1}))\leq\tau C\big((h^l)^2+\tau^2\big) +\frac{\kappa}{4\tau}\Vert L_h(\delta_t e_\phi^{n+1})\Vert^2_{H^1},\\		
		&-\tau\tilde R^{n+1}_c(e_c^{n+1}-e_\phi^n,\nabla (e_c^{n+1}-e_\phi^n))\leq C\tau \big(h^{2l}+\tau^2\big)+\tau\kappa\Vert(e_c^{n+1}-e_\phi^n)\Vert^2_{H^1(\Omega)},\\
		&\tilde R^{n+1}_\mu(\delta_t e_\phi^{n+1})\leq \tau C\big((h^l)^2+\tau^2\big) +\frac{\kappa}{4\tau}\Vert \nabla L_h(\delta_t e_\phi^{n+1})\Vert^2.		
	\end{align*} Here $\kappa>0$ is any positive constant.
\end{lemma}
\begin{proof}
	Replacing $\xi$ in definition \eqref{rem.til.phi} with $e_\phi^{n+1}$, we have
	\begin{align*}
		&\tau\tilde R^{n+1}_\phi(e_\phi^{n+1},\nabla  e_\phi^{n+1})= \tau\left\langle \frac{\phi_h(t_{n+1})-\phi_h(t_{n})}{\tau}
		-\phi_t(t_{n+1}),e_\phi^{n+1}\right\rangle\\&
		-\tau\Big\langle-c(t_{n+1})\nabla \big(c(t_{n+1})-\phi(t_{n+1})\big)- \big(-c_h(t_n)\nabla \big(c_h(t_{n+1})-\phi_h(t_n)\big)\big),\nabla e_\phi^{n+1} \Big\rangle\\
		=& \tau\left\langle \frac{\phi_h(t_{n+1})-\phi_h(t_{n})-\delta\phi(t_{n+1})}{\tau}+\frac{\delta\phi(t_{n+1})}{\tau} - \phi_t(t_{n+1}),e_\phi^{n+1}\right\rangle \\
		&-\tau\Big\langle-\left(c(t_{n+1})\nabla (c(t_{n+1})-\phi(t_{n+1}))-c(t_{n})\nabla (c(t_{n+1})-\phi(t_{n+1}))\right)\\
		&-(c(t_{n})\nabla (c(t_{n+1})-\phi(t_{n+1}))-c(t_{n})\nabla (c(t_{n+1})-\phi(t_{n})))\\
		&-(c(t_{n})\nabla (c(t_{n+1})-\phi(t_{n}))-c_h(t_{n})\nabla (c(t_{n+1})-\phi(t_{n})))\\
		& - \big(-c_h(t_n)\nabla \big(c_h(t_{n+1})-\phi_h(t_n)-c(t_{n+1})+\phi(t_{n})\big)\big),\nabla e_\phi^{n+1} \Big\rangle\\
		=&\tau\left\langle \Delta_1+\Delta_2,e_\phi^{n+1}\right\rangle + \tau\Big\langle \Delta_3 +\Delta_4 +\Delta_5 + \Delta_6,\nabla e_\phi^{n+1} \Big\rangle.
	\end{align*}
	Here,
	\begin{align*}
		&\Vert\Delta_1\Vert =  \bigg\Vert\frac{\phi_h(t_{n+1})-\phi_h(t_{n})-\delta\phi(t_{n+1})}{\tau}\bigg\Vert =O(h^{1+l}),\\
		&\Vert\Delta_2\Vert =  \bigg\Vert\frac{\delta\phi(t_{n+1})}{\tau} - \phi_t(t_{n+1})\bigg\Vert=O(\tau),\\
		&\Vert\Delta_3 \Vert= \Vert c(t_{n+1})\nabla (c(t_{n+1})-\phi(t_{n+1}))-c(t_{n})\nabla (c(t_{n+1})-\phi(t_{n+1})) \Vert= O(\tau),\\
		&\Vert\Delta_4 \Vert= \Vert c(t_{n})\nabla (c(t_{n+1})-\phi(t_{n+1}))-c(t_{n})\nabla (c(t_{n+1})-\phi(t_{n}))\Vert=O(\tau),\\
		&\Vert\Delta_5\Vert = \Vert c(t_{n})\nabla (c(t_{n+1})-\phi(t_{n}))-c_h(t_{n})\nabla (c(t_{n+1})-\phi(t_{n}))\Vert=O(h^{1+l}),\\
		&\Vert\Delta_6\Vert = \Vert-c_h(t_n)\nabla \big(c_h(t_{n+1})-\phi_h(t_n)-c(t_{n+1})+\phi(t_{n})\big)\Vert=O(h^{1+l}).
	\end{align*}
	By the Cauchy-Schwarz inequality and the Young's inequality and using the assumption \eqref{re-reg}, we have
	\begin{align}
		-\tau\tilde R^{n+1}_\phi(e_\phi^{n+1},\nabla  e_\phi^{n+1})&\leq \frac{\tau}{4}(\sum_{j=1}^6\Vert\Delta_j\Vert^2)+
		\tau\left(\Vert e_\phi^{n+1}\Vert^2+\Vert \nabla  e_\phi^{n+1}\Vert^2\right)\leq\tau C\big((h^l)^2+\tau^2\big)+\tau\Vert e_\phi^{n+1}\Vert_{H^1}^2.\notag
	\end{align}
	Here, we use \eqref{pr-reg-1} and the assumption \eqref{re-reg}(For $\Delta_1$, we need to take the partial derivative of \eqref{1.pro} concerning $t$ and then use the assumption \eqref{pr-reg-1}, which is similar to \eqref{phi.av}). Analogous estimates can be established
	:
	\begin{align*}
	-\tilde R^{n+1}_\phi(L_h(\delta_t e_\phi^{n+1}),\nabla L_h(\delta_t e_\phi^{n+1}))&\leq C\tau(\sum_{j=1}^6\Vert\Delta_j\Vert^2)+\frac\kappa{4\tau}
		\left(\Vert L_h(\delta_t e_\phi^{n+1})\Vert^2+\Vert \nabla  L_h(\delta_t e_\phi^{n+1})\Vert^2\right)\notag\\
		&\leq\tau C\big((h^l)^2+\tau^2\big)+\frac\kappa{4\tau}\Vert L_h(\delta_t e_\phi^{n+1})\Vert_{H^1}^2\notag
		\end{align*}
	and
	\begin{align*}
		&-\tau\tilde R_c(e_c^{n+1}-e_\phi^n,\nabla (e_c^{n+1}-e_\phi^n))=-\Big(\tau\left\langle  \frac{c_h(t_{n+1})-c_h(t_{n})}{\tau} - c_t(t_{n+1}),e_c^{n+1}-e_\phi^n\right\rangle. \\
		&+ \tau \Big\langle  c(t_{n+1})\nabla \mu(t_{n+1})-2c^2(t_{n+1})\nabla \big(c(t_{n+1})-\phi(t_{n+1})\big) \\
		&-\big(c(t_{n})\nabla \mu_h(t_{n+1})-2c_h^2(t_n)\nabla \big(c_h(t_{n+1})-\phi_h(t_n)\big)\big),\nabla (e_c^{n+1}-e_\phi^n)\Big\rangle\\
		&+\tau\left\langle \nabla \big(\phi(t_{n+1})-\phi_h(t_n)\big),\nabla (e_c^{n+1}-e_\phi^n)\right\rangle\Big),\\
		&\leq \tau C\big((h^l)^2+\tau^2\big)+\tau\kappa(\Vert(e_c^{n+1}-e_\phi^n)\Vert^2+\Vert\nabla (e_c^{n+1}-e_\phi^n)\Vert^2)
	\end{align*}
	Next, we control $-\tau\tilde R^{n+1}_\phi(e_\mu^{n+1},\nabla  e_\mu^{n+1})$ by using the estimate $\tau\Vert \nabla e_\mu^{n+1}-c^n\nabla(e_c^{n+1}-e_\phi^{n})\Vert^2$ and the estimate \eqref{basic-3-1} in Lemma \ref{lemma-est-1} for $\kappa\tau\Vert c^n\nabla(e_c^{n+1}-e_\phi^{n+1})\Vert^\frac43_{L^{\frac65}}$.
	In fact, the following estimate is established:
	\begin{align}
		&-\tau\tilde R^{n+1}_\phi(e_\mu^{n+1},\nabla  e_\mu^{n+1})\leq \tau\sum_{j=1}^2\Vert\Delta_j\Vert_{L^6}\Vert e_\mu^{n+1}\Vert_{L^{\frac65}}+\tau\sum_{j=3}^6\Vert\Delta_j\Vert_{L^6}\Vert \nabla e_\mu^{n+1}\Vert_{L^{\frac65}}\notag\\
		\leq &\tau\sum_{j=1}^2\Vert\Delta_j\Vert_{L^6}(\Vert \overline{e_\mu^{n+1}}\Vert_{L^{\frac65}}+\Vert e_\mu^{n+1}-\overline{e_\mu^{n+1}}\Vert_{L^{\frac65}})+\tau\sum_{j=3}^6\Vert\Delta_j\Vert_{L^6}\Vert \nabla e_\mu^{n+1}\Vert_{L^{\frac65}}\notag\\
		\leq &\tau\sum_{j=1}^2\Vert\Delta_j\Vert_{L^6}\Vert \overline{e_\mu^{n+1}}\Vert_{L^{\frac65}}+\tau\sum_{j=3}^6\Vert\Delta_j\Vert_{L^6}\big(\Vert(\nabla e_\mu^{n+1}-c^n\nabla(e_c^{n+1}-e_\phi^{n+1}))\Vert_{L^{\frac65}}+\Vert c^n\nabla(e_c^{n+1}-e_\phi^{n+1})\Vert_{L^{\frac65}}\big)\notag\\
		\leq & C\tau\big(h^{2l}+\tau^2+\Vert e_\phi^n\Vert^2\big)+\kappa\tau\Vert \nabla e_\mu^{n+1}-c^n\nabla(e_c^{n+1}-e_\phi^{n+1})\Vert^2_{L^{\frac65}}\nonumber\\
		&+\frac{C\tau}{(\tau^2+h^{2l})}\sum_{j=1}^6\Vert\Delta_j\Vert^4_{L^6}+\kappa\tau(\tau^2+h^{2l})^{\frac13}\Vert c^n\nabla(e_c^{n+1}-e_\phi^{n+1})\Vert^\frac43_{L^{\frac65}}\nonumber\\
		\leq & C\tau\big(h^{2l}+\tau^2+\Vert e_\phi^n\Vert^2\big)+\kappa\tau\Vert \nabla e_\mu^{n+1}-c^n\nabla(e_c^{n+1}-e_\phi^{n})\Vert^2+\kappa\tau\Vert\nabla(e_c^{n+1}-e_\phi^{n})\Vert^2+C\kappa (h^{2l}+\tau^2)\tau\|c^n\|_{H^1}^2\nonumber
	\end{align} for any $\kappa>0$. Here, we used Lemmas \ref{lem.avg}-\ref{lemma-est-1}, the Poincare's inequality and the Young's inequality $xy\le \kappa_1 x^{\frac43}+\kappa_1^{-3}y^4$ for any $\kappa_1>0$.
	
	The remain is to estimate $\tilde R_\mu(\delta_t e_\phi^{n+1})$as follows. Replacing $\xi$ in \eqref{rem.til.mu} with $\delta_t e_\phi^{n+1}$ and using \eqref{apf-2}, the assumption \eqref{re-reg}, \eqref{def.iap},\eqref{av-1} and the Cauchy-Schwarz's inequality, 
	we have
	\begin{align*}
		&\tilde R^{n+1}_\mu(\delta_t e_\phi^{n+1})=\frac{1}{\varepsilon^2}\left\langle f(\phi(t_{n+1}))-f(\phi_h(t_{n}))-S(\phi_h(t_{n+1})-\phi_h(t_n)),\delta_t e_\phi^{n+1}\right\rangle \\
		&- \left\langle \mu(t_{n+1}) - \mu_h(t_{n+1}) + c(t_{n+1})- c_h(t_{n+1}),\delta e_\phi^{n+1}\right\rangle\\
		&=\frac{1}{\varepsilon^2}\left\langle \nabla(f(\phi(t_{n+1}))-f(\phi_h(t_{n}))-S(\phi_h(t_{n+1})-\phi_h(t_n))),\nabla L_h(\delta_t e_\phi^{n+1})\right\rangle\\
		&-\frac{1}{\varepsilon^2}\left\langle f(\phi(t_{n+1}))-f(\phi_h(t_{n}))-S(\phi_h(t_{n+1})-\phi_h(t_n)),1\right\rangle\left\langle \delta_t e_\phi^{n+1},1\right\rangle \\
		&- \left\langle \nabla(\mu(t_{n+1}) - \mu_h(t_{n+1}) + c(t_{n+1})- c_h(t_{n+1})),\nabla L_h(\delta_t e_\phi^{n+1})\right\rangle\\
		&+\left\langle \mu(t_{n+1}) - \mu_h(t_{n+1}) + c(t_{n+1})- c_h(t_{n+1}),1\right\rangle\left\langle \delta_t e_\phi^{n+1},1\right\rangle\\
		&\le \tau C\big((h^l)^2+\tau^2\big) +\frac{\kappa}{4\tau}\Vert \nabla L_h(\delta_t e_\phi^{n+1})\Vert^2.
	\end{align*}
	This completes the proof of Lemma \ref{lem.est.til}.
\end{proof}

For the remainder from numerical approximation on the right-hand side of \eqref{3.est.f}, we have the following estimates
\begin{lemma}\label{lem.est.or}
	Under assumption \eqref{apf-2}, \eqref{re-reg} and \eqref{pr-reg-1}-\eqref{pr-reg-3}, we have
	\begin{align*}
		&\tau R^{n+1}_\phi(\nabla e_\phi^{n+1})\leq\tau C\Vert e_c^n\Vert^2+\tau\kappa\Vert \nabla e_\phi^{n+1}\Vert^2,\,R^{n+1}_\phi(\nabla L_h(\delta_t e_\phi^{n+1}))\leq\tau C\Vert e_c^n\Vert^2+\frac\kappa{4\tau}\Vert\nabla L_h(\delta_t e_\phi^{n+1})\Vert^2,\\
		&\tau R^{n+1}_\phi(\nabla e_\mu^{n+1}) +\tau R^{n+1}_c(\nabla(e_c^{n+1}-e_\phi^n))\leq C\tau\Vert e_c^n\Vert^2+\kappa\tau\Vert \nabla e_\mu^{n+1}-c^n\nabla(e_c^{n+1}-e_\phi^n)\Vert^2,\\
		&-R_\mu^{n+1}(\delta_t e_\phi^{n+1})\leq\tau C\big((h^l)^2+\tau^2\big) +\frac{\kappa}{4\tau}\Vert \nabla L_h(\delta_t e_\phi^{n+1})\Vert^2.
	\end{align*}
\end{lemma}
\begin{proof}
	Similar to Lemma \ref{lem.est.til}, the proof is a result of using Assumption \eqref{re-reg}, \eqref{pr-reg-1}-\eqref{pr-reg-3}, the Schwarz's inequality, and the Young's inequality, so the details are omitted. Here we outline them as follows.
	
	Replacing $\xi$ in definition \eqref{rem.phi} by $e_\phi^{n+1}, L_h(\delta_t e_\phi^{n+1})$, $e_\mu^{n+1}$, and $\eta$ in definition \eqref{rem.c} by $e_c^{n+1}-e_\phi^n$ correspondingly, we have 
	\begin{align*}
		&\tau R^{n+1}_\phi(\nabla e_\phi^{n+1})= \tau\left\langle -e_c^{n}\nabla \big(c_h(t_{n+1})-\phi_h(t_n)\big),\nabla e_\phi^{n+1}\right\rangle\leq\tau C\Vert e_c^n\Vert^2+\kappa\tau\Vert \nabla e_\phi^{n+1}\Vert^2,\\
		&R^{n+1}_\phi(\nabla L_h(\delta_t e_\phi^{n+1}))=\left\langle -e_c^{n}\nabla \big(c_h(t_{n+1})-\phi_h(t_n)\big),\nabla L_h(\delta_t e_\phi^{n+1})\right\rangle\leq\tau C\Vert e_c^n\Vert^2+\frac{\kappa}{4\tau}\Vert\nabla  L_h(\delta_t e_\phi^{n+1})\Vert^2
	\end{align*}
	and
	\begin{align}
		&\tau R^{n+1}_\phi(\nabla e_\mu^{n+1}) +\tau R^{n+1}_c(\nabla(e_c^{n+1}-e_\phi^n))=\tau \left\langle e_c^n\nabla \big(c_h(t_{n+1})-\phi_h(t_n)\big),\nabla e_\mu^{n+1}\right\rangle\nonumber\\
		&\phantom{xx}{} +\tau \big\langle \nabla \mu_h(t_{n+1})- (c^n+c_h(t_n))\nabla \big(c_h(t_{n+1})-\phi_h(t_n)\big), e_c^n\nabla(e_c^{n+1}-e_\phi^n)\big\rangle\nonumber\\
		&=\tau \big\langle e_c^n\nabla \mu_h(t_{n+1})- e_c^nc_h(t_n)\nabla \big(c_h(t_{n+1})-\phi_h(t_n)\big), \nabla(e_c^{n+1}-e_\phi^n)\big\rangle\nonumber\\
		&\phantom{xx}{} +\tau\left\langle e_c^n\nabla \big(c_h(t_{n+1})-\phi_h(t_n)\big),\nabla e_\mu^{n+1}-c^n\nabla(e_c^{n+1}-e_\phi^n)\right\rangle\nonumber\\
		&\leq C\tau\Vert e_c^n\Vert^2+\kappa\tau\Vert\nabla(e_c^{n+1}-e_\phi^n)\Vert^2+\kappa\tau\Vert \nabla e_\mu^{n+1}-c^n\nabla(e_c^{n+1}-e_\phi^n)\Vert^2.\nonumber
	\end{align}
	Notice that there exists one cancellation between the term $\tau R^{n+1}_\phi(\nabla e_\mu^{n+1})$ and $\tau R^{n+1}_c(\nabla(e_c^{n+1}-e_\phi^n))$. This is key for our success in establishing the desired estimates for this two remainders.
	
	Also, replacing $\sigma$ in definition \eqref{rem.mu} by $\delta_t e_\phi^{n+1}$, we have
	\begin{align*}
		&-R^{n+1}_\mu(\delta_t e_\phi^{n+1})
		=-\frac{1}{\varepsilon^2}\left\langle f(\phi^n)-f(\phi_h(t_n)),\delta e_\phi^{n+1}\right\rangle=-\frac{1}{\varepsilon^2}\left\langle \nabla(f(\phi^n)-f(\phi_h(t_n))),\nabla L_h(\delta e_\phi^{n+1})\right\rangle\\&+\frac{1}{\varepsilon^2}\left\langle f(\phi^n)-f(\phi_h(t_n)),1\right\rangle\left\langle\delta e_\phi^{n+1},1\right\rangle\le \tau C\big((h^l)^2+\tau^2\big) +\frac{\kappa}{4\tau}\Vert \nabla L_h(\delta e_\phi^{n+1})\Vert^2.
	\end{align*}
	
	
	This completes the proof of Lemma \ref{lem.est.or}.\end{proof}


\noindent\textbf{The proof of Theorem \ref{thm.od}}.
It follows from \eqref{3.est.f}, together with Lemma \ref{lem.est.til} and Lemma \ref{lem.est.or} and by the Poincare's inequality $\Vert v -\bar v\Vert_{L^2(\Omega)}\leq C_\Omega\Vert \nabla  v\Vert_{L^2(\Omega)}$ for all $v\in H^1(\Omega)$, that
\begin{align}
	&\frac{1}2(\Vert e_\phi^{n+1}\Vert^2-\Vert e_\phi^n\Vert^2+\Vert\delta_t e_\phi^{n+1}\Vert^2)+\frac{C-\kappa}{\tau}\| L_h(\delta_t e_\phi^{n+1})\|_{H^1}^2+\frac{1}{2}(\Vert e_c^{n+1}\Vert^2-\Vert e_c^n\Vert^2+\Vert\delta_t e_c^{n+1}\Vert^2)\nonumber\\
	&+\left\langle e_c^n,e_\phi^n\right\rangle+\frac{1}2(\Vert \nabla  e_\phi^{n+1}\Vert^2-\Vert \nabla e_\phi^n\Vert^2+\Vert\nabla (\delta_t e_\phi^{n+1})\Vert^2)+\tau(1-3\kappa)\Vert\nabla (e_c^{n+1}-e_\phi^{n})\Vert^2\nonumber\\
	&\leq-\tau(\frac{1}4-2\kappa)\Vert (\nabla e_\mu^{n+1}- c^n\nabla (e_c^{n+1}-e_\phi^{n}))\Vert^2-\frac{S}{\varepsilon^2}\Vert\delta_t e_\phi^{n+1}\Vert^2
	+\left\langle e_c^{n+1},e_\phi^{n+1}\right\rangle \label{thm-ine-100}\\
	&+C\tau \big((h^{2l}+\tau^2)+ (h^{2l}+\tau^2)\|c^n\|_{H^1}^2+\Vert e_\phi^n\Vert_{H^1}^2+\Vert e_\phi^{n+1}\Vert_{H^1}^2+\Vert\delta_t e_\phi^{n+1}\Vert^2
	+\Vert e_c^n\Vert^2+\Vert e_c^{n+1}\Vert^2 \big).\nonumber
\end{align}

Choosing $\kappa>0$ to be small sufficiently, changing the superscript of the error functions to $k$, and summing up the inequality \eqref{thm-ine-100} from $k=0$ to $n(0\leq n\leq N-1)$, with the help of the fact that $\tau\sum_0^n\|c^k\|^2_{H^1}\le C$, we have
\begin{align*}
	&\frac{1}2\Vert e_\phi^{n+1}\Vert^2+\frac{1}2\Vert e_c^{n+1}\Vert^2+\frac{1}2\Vert \nabla e_\phi^{n+1}\Vert^2-\left\langle e_c^{n+1}, e_\phi^{n+1}\right\rangle\\
	&+\left(\frac{1}4-2\kappa\right)\sum_{k=0}^{n}\tau\Vert\nabla e_\mu^{n+1}- c^n\nabla (e_c^{n+1}-e_\phi^{n})\Vert^2+(1-3\kappa)\sum_{k=0}^{n}\tau\Vert\nabla (e_c^{k+1}-e_\phi^{k})\Vert^2\\
	\leq &\tau C\sum_{k=0}^n(h^{2l}+\tau^2+\Vert e_\phi^k\Vert_{H^1}^2+\Vert e_\phi^{k+1}\Vert_{H^1}^2+\Vert \delta e_\phi^{k+1}\Vert^2+ \Vert e_c^k\Vert^2+\Vert e_c^{k+1}\Vert^2)\\
	&\quad+\frac{1}2\Vert e_\phi^{0}\Vert^2+\frac{1}2\Vert e_c^{0}\Vert^2+\frac{1}2\Vert \nabla e_\phi^{0}\Vert^2-\left\langle e_c^{0}, e_\phi^{0}\right\rangle+C (h^{2l}+\tau^2),
\end{align*}
which, together with the fact that $\|e_c^{n+1}\|^2\le C(\|e_c^{n+1}-e_\phi^{n+1}\|^2+\|e_\phi^{n+1}\|^2)$ and that $\vert\overline{e_{\phi}^{n+1}}\vert\leq C(\tau+h^l)$ given by Lemma \ref{lem.avg}, with the help of the Poincare's inequality, yields that
\begin{align}
	&\left(\frac{1}2-\tau C\right)\Vert e_\phi^{n+1}-e_c^{n+1}\Vert^2+\left(\frac{1}2-\tau C\right)\Vert \nabla e_\phi^{n+1}\Vert^2+(1-3\kappa)\sum_{k=0}^{n}\tau\Vert\nabla (e_c^{k+1}-e_\phi^{k})\Vert^2\nonumber\\
	&+\left(\frac{1}4-2\kappa\right)\sum_{k=0}^{n}\tau\Vert\nabla e_\mu^{n+1}- c^n\nabla (e_c^{n+1}-e_\phi^{n})\Vert^2\nonumber\\
	&\leq\tau C\sum_{k=0}^n\big(h^{2l}+\tau^2+\Vert \nabla e_\phi^k\Vert^2+ \Vert e_\phi^k\Vert^2+\Vert e_\phi^{k}-e_c^{k}\Vert^2\big)+C (h^{2l}+\tau^2).\label{energ-1}
\end{align}

According to the Poincare's inequality, 
it follows from \eqref{energ-1} that there exist a $\tau^*>0$ such that for some positive constant $C'$ (independent of $\tau,\, h$) and for all $\tau\leq \tau^*$
\begin{align}
	&\Vert e_\phi^{n+1}-e_c^{n+1}\Vert^2+\Vert e_\phi^{n+1}\Vert_{H^1}^2+\sum_{k=0}^{n}\tau\Vert\nabla e_c^{k+1}\Vert^2+\sum_{k=0}^{n}\tau\Vert\nabla e_\mu^{n+1}- c^n\nabla (e_c^{n+1}-e_\phi^{n})\Vert^2\nonumber\\
	&\leq C'(h^{l}+\tau)^2+\tau C'\sum_{k=0}^n\big(\Vert e_\phi^k\Vert_{H^1}^2+\Vert e_\phi^{k}-e_c^{k}\Vert^2\big)\label{est.main0}
\end{align}
holds for all $n=0,1,\cdots,N-1$. 
Applying the Gr\"onwall's inequality to \eqref{est.main0}, we have 
\begin{align}\label{est.main}
	\Vert e_\phi^{n+1}\Vert_{H^1}+\Vert e_\phi^{n+1}-e_c^{n+1}\Vert\le C'\exp(C'T)(h^l+\tau).
\end{align}
Moreover, substituting \eqref{est.main} into \eqref{est.main0}, we can get the following estimates on $\Vert e_c^n\Vert_{l^2(H^1)}$
\begin{align}\Vert e_c^{n+1}\Vert^2+\sum_{k=0}^{n}\tau\Vert\nabla e_c^{k+1}\Vert^2\le C(h^l+\tau)^2.\label{est.main-c}\end{align}
Similarly, we have
\begin{align}
	&\sum_{k=0}^{n}\tau^2\Vert\nabla e_\mu^{k+1}\Vert^\frac43_{L^\frac65}\le C\sum_{k=0}^{n}\tau^2\left(\Vert\nabla e_\mu^{k+1}- c^k\nabla (e_c^{k+1}-e_\phi^{k})\Vert^\frac43_{L^\frac65}+\Vert c^k\nabla (e_c^{k+1}-e_\phi^{k})\Vert^\frac43_{L^\frac65}\right)\nonumber\\
	&\leq C\left(\sum_{k=0}^{n}\tau^3\right)^\frac13\left(\sum_{k=0}^{n}\tau^\frac32\Vert\nabla e_\mu^{k+1}- c^k\nabla (e_c^{k+1}-e_\phi^{k})\Vert^2_{L^\frac65}\right)^\frac23\noN\\&
	+C\sum_{k=0}^{n}\left(\tau^3\tau\|c^k\|_{H^1}^2+\tau\Vert \nabla (e_c^{k+1}-e_\phi^{k})\Vert^2\right)\nonumber\\
	&\leq C\tau\left(\sum_{k=0}^{n}\tau\Vert\nabla e_\mu^{k+1}- c^k\nabla (e_c^{k+1}-e_\phi^{k})\Vert^2\right)^\frac23+C\sum_{k=0}^{n}\left(\tau^3\tau\|c^k\|_{H^1}^2+\tau\Vert\nabla (e_c^{k+1}-e_\phi^{k})\Vert^2\right)\nonumber\\
	&\le C\tau \left(h^l+\tau\right)^{\frac43}+C\tau^3+\left(h^l+\tau\right)^2\nonumber.
\end{align}
Thus, assuming $h^l\le \tau$, we have
\begin{align}
	&\sum_{k=0}^{n}\tau\Vert\nabla e_\mu^{k+1}\Vert^\frac43_{L^\frac65} \le C\left(h^l+\tau+\frac{h^{2l}}\tau\right)\leq C\left(h^l+\tau\right).\label{est.main-G}
\end{align}
Combining estimates \eqref{est.main},\eqref{est.main-c} and \eqref{est.main-G}, we get the desired estimate \eqref{est.main-S}. And then the estimate \eqref{est.main-SS} is obtained by using \eqref{pr-reg-1}-\eqref{est.main-S}.

This completes the proof of Theorem \ref{thm.od}.

\section{Convergence analysis}\label{sec-can}
This section establishes that, as $\tau, h \to 0$, the numerical solution of the discrete scheme \eqref{2.phi}--\eqref{2.mu} converges to a weak solution of the continuous model \eqref{1.mu}, even in the absence of the higher regularity assumptions on the solution that were previously required in Section~\ref{sec-er}. Our analysis is inspired by the foundational work of \cite{CE1992}, where a rigorous convergence framework was developed for the Cahn–Hilliard equation with a logarithmic free energy. In contrast to their setting, we do not impose strong regularity assumptions on the solution itself—only mild regularity on the initial data $(\phi_0, c_0)$ is required.

A key innovation in this section lies in the derivation of a novel a priori bound for the discrete chemical potential sequence. Specifically, we obtain an $L^{\frac{4}{3}}(0,T; W^{1,\frac{6}{5}}(\Omega))$ estimate for the numerical chemical potential $\mu^{n+1}$ associated with the time-discrete scheme. This estimate is crucial, as it compensates for the lack of a uniform-in-time $L^2(0,T; H^1(\Omega))$ bound for $\mu^{n+1}$, which is no longer attainable due to the presence of the nonlinear cross-diffusion term in \eqref{1.mu}. This challenge distinguishes our system from the classical Cahn–Hilliard model and necessitates the development of refined analytical tools. 

Define
\begin{equation*}
	\tilde\xi(t) =\frac{t-t_n}{t_{n+1}-t_n}\xi^{n+1}+\frac{t_{n+1}-t}{t_{n+1}-t_n}\xi^n,\,t\in[t_n,t_{n+1}), n=0,\cdots, N-1
\end{equation*}
and
\begin{equation*}
	\xi^-(t) = \xi^n,\,t\in(t_n,t_{n+1}], n=0,\cdots, N-1,\,\xi^+(t) = \xi^{n+1},\,t\in(t_n,t_{n+1}], n=0,\cdots, N-1.
\end{equation*}

The main result of this section on convergence analysis reads as follows.
\begin{thm}\label{thm.cov} Denote $X=W^{1,4}(\Omega)$ when $d=1,2$ and $X=W^{1,6}(\Omega)$ when $d=3$ and $X'$ is the dual space of $X$. Let $S>\frac L2$, $(\phi^n,c^n,\mu^n)$ be the solution to the numerical scheme \eqref{2.phi}-\eqref{2.mu}. Assume that $\phi_0,c_0\in H^1(\Omega)$ and let the assumptions \eqref{apf-1}-\eqref{apf-2} and the initial bound \eqref{er-in} hold. Then there exist functions $(\phi^*,c^*,\mu^*)\in L^\infty([0,T]; H^1(\Omega))\times L^2([0,T]; L^2(\Omega))\times L^2([0,T];H^1(\Omega))$ with $(\partial_t\phi^*,\partial_tc^*)\in L^2([0,T]; (H^1(\Omega))')\times L^{4/3}([0,T]; X')$ such that the following convergence hold, when $h\to 0$ and $\tau\to 0$,
	\begin{align}
		\Vert\tilde\phi -\phi^\pm&\Vert_{L^2([0,T];H^1(\Omega))}\to 0,\quad \Vert\tilde c -c^\pm\Vert_{L^2(0,T;L^2)}\to 0,\label{cog-1}\\
		\tilde\phi,\phi^\pm&\rightarrow \phi^{*} \text{ weakly }^*\text{ in } {L^\infty([0,T];H^1(\Omega))},\label{cog-2}\\
		\partial_t\tilde\phi&\rightarrow \partial_t\phi^{*} \text{ weakly in } {L^2([0,T];(H^1(\Omega))')},\label{cog-3}\\
		\tilde\phi,\phi^\pm&\rightarrow \phi^{*}\text{ strongly in } {L^2([0,T];L^s(\Omega))},\quad 2\le s<6,\label{cog-4}\\
		\tilde c,c^\pm&\rightarrow c^{*} \text{ weakly }^*\text{ in } {L^\infty([0,T];L^2(\Omega))},\label{cog-5}\\
		\tilde c,c^\pm&\rightarrow c^{*} \text{ weakly } \text{ in } {L^2([0,T];H^1(\Omega))},\label{cog-6}\\
		\partial_t\tilde c&\rightarrow \partial_tc^{*} \text{ weakly }\text{ in } {L^{\frac43}([0,T];X')},\label{cog-7}\\
		\tilde c,c^\pm&\rightarrow c^{*}\text{ strongly in } {L^2([0,T];L^2(\Omega))},\label{cog-8}\\
		c^+-\phi^-&\rightarrow c^{*}-\phi^{*} \text{ weakly }\text{ in } {L^2([0,T];H^1(\Omega))},\label{cog-9}\\
		\mu^+&\rightarrow \mu^{*} \text{ weakly }\text{ in } {L^\frac43([0,T];L^{\frac65}(\Omega))},\label{cog-10}\\
		\nabla \mu^+-c^-&\nabla (c^+-\phi^-)\rightarrow \nabla \mu^{*}-c^*\nabla (c^*-\phi^*) \text{ weakly }\text{ in } {L^\frac43([0,T];L^\frac65(\Omega))}\label{cog-11}
	\end{align}
	and $(\phi^*,c^*,\mu^*)$ is the solution of the system \eqref{1.mu} -\eqref{1.ic} in the sense of \begin{align}
		\int_0^T\bigg(\left\langle \partial_t\phi^*,\xi \right\rangle,\chi\bigg) ds =& \int_0^T\bigg(-\left\langle \nabla \mu^*-c^*\nabla \big(c^*-\phi^*\big),\nabla \xi\right\rangle,\chi\bigg)ds,\label{sou-1}\\
		\int_0^T\bigg(\left\langle \partial_tc^*,\eta\right\rangle,\zeta\bigg)ds =& \int_0^T\bigg(\left\langle c^-\nabla \mu^*-(c^*)^2\nabla \big(c^{*}-\phi^*\big)- \nabla \big(c^{*}-\phi^*\big),\nabla \eta\right\rangle,\zeta\bigg)ds, \label{sou-2}\\
		\int_0^T\bigg(\langle\mu^{*},\sigma\rangle,\upsilon\bigg)ds =& \int_0^T\bigg(\left\langle \nabla \phi^{*}, \nabla \sigma\right\rangle,\upsilon\bigg)ds+\int_0^T\bigg(\left\langle \frac1{\varepsilon^2}f(\phi^{*})- c^{*},\sigma\right\rangle,\upsilon\bigg)ds,\label{sou-3}
	\end{align} for any $\xi(x)\in W^{1,6}(\Omega), \eta(x)\in W^{1,6}(\Omega), \sigma(x)\in H^1(\Omega)$ and any $\chi(t)\in L^4(0,T), \zeta(t)\in L^4(0,T), \upsilon(t)\in L^4(0,T)$. Moreover, $(\phi^*, c^*)(t=0)=(\phi_0(x), c_0(x))$ holds in the sense of $(H^1(\Omega))'\times X'$.
\end{thm}
\begin{proof}
	Firstly, under the assumptions of Theorem \ref{thm.cov}, it is easy to know that the basic estimate \eqref{basic-1} holds for some positive constant $C$. Now we establish an estimate of $\mu$ uniformly in $h>0$ and $\tau>0$. In fact, on one hand, it is easy to get that
	\begin{align}
		&\sum_{k=0}^{N}\tau\Vert \nabla \mu^{k+1}\Vert_{L^\frac65}^\frac43\le C\sum_{k=0}^{N}\tau(\Vert\nabla \mu^{k+1}-c^k\nabla (c^{k+1}-\phi^k)\Vert^\frac43+\Vert c^k\nabla (c^{k+1}-\phi^k)\Vert^\frac43_{L^\frac65})\nonumber\\
		&\le C(T)\left(\sum_{k=0}^{N}\tau\Vert\nabla \mu^{k+1}-c^k\nabla (c^{k+1}-\phi^k)\Vert^2\right)^\frac23+C\sum_{k=0}^{N}\tau\Vert c^k\nabla (c^{k+1}-\phi^k)\Vert^\frac43_{L^\frac65}.\nonumber
	\end{align}
	On the other hand, according to the H\"older's inequality, we have
	\begin{align}
		\sum_{k=0}^{N}\tau\Vert c^k\nabla (c^{k+1}-\phi^k)\Vert_{L^\frac65}^{\frac43}
		\le \left(\sum_{k=0}^{N}\tau \left(\int_\Omega \vert c^k\vert^3 dx\right)^\frac{4}3\right)^\frac13\left(\sum_{k=0}^{N}\tau \Vert \nabla (c^{k+1}-\phi^k)\Vert^2 \right)^\frac23\nonumber
	\end{align}
	and
	\begin{align}
		&\left(\sum_{k=0}^{N}\tau\left(\int_\Omega \vert c^k\vert^3 dx\right)^\frac{4}3\right)^\frac14\le\left(\sum_{k=0}^{N}\tau \left(\int_\Omega \vert c^k\vert^2 dx\right)^\frac{2}3\left(\int_\Omega \vert c^k\vert^4 dx\right)^\frac{2}3\right)^\frac14\noN\\
		&\le\max_{0\le k\le N}\Vert c^k\Vert^\frac12\left(\sum_{k=0}^{N}\tau \Vert c^k\Vert^2_{H_1}\right)^\frac14.\label{basic-100}
	\end{align}
	
	Hence, recalling \eqref{2.phi}, \eqref{2.mu} and using estimates \eqref{basic-1}-\eqref{basic-2a} in Lemma \ref{lemma-est}, by the Poincare's inequality and the Young's inequality, we have
	\begin{align}
		&\sum_{k=0}^N\tau\Vert\mu^{n+1}\Vert^\frac43_{L^\frac65} \leq \sum_{k=0}^N\tau\Vert\overline{\mu^{n+1}}\Vert^\frac43_{L^\frac65} +C\sum_{k=0}^N\tau\Vert\nabla\mu^{n+1}\Vert^\frac43_{L^\frac65}\leq C\label{basic-2}
	\end{align}
	and
	\begin{align}
		&\Vert\phi^{n+1}\Vert^2+\sum_{k=0}^{n-1}\Vert\delta_t\phi^{k+1}\Vert^2 \leq \sum_{k=0}^n\tau\Vert\nabla \mu^{k+1}-c^k\nabla (c^{k+1}-\phi^k)\Vert^2+C(n+1)\tau\leq C. \label{basic-3}
	\end{align}
	Thus, using the assumption \eqref{apf-2} and by the Poincare's inequality and inequalities \eqref{basic-1}, \eqref{basic-100}, \eqref{basic-2}-\eqref{basic-3}, we have, for $n\in 0,\cdots,N-1$,
	\begin{align}
		&\max_{n=0,\cdots,N-1}\left\{\Vert\phi^{n+1}\Vert_{H^1}^2 + \Vert c^{n+1}\Vert^2\right\}+\,\sum_{k=0}^{n}\Vert \delta_t c^{k+1}\Vert^2+\sum_{k=0}^n\Vert(\delta_t\phi^{k+1})\Vert_{H^1}^2\leq C,\label{es-1}\\
		&\sum_{k=0}^n\tau\Vert\nabla (c^{k+1}-\phi^k)\Vert^2+\sum_{k=0}^n\tau\Vert c^{k+1}\Vert^2_{H^1}+\sum_{k=0}^n\tau\Vert c^{k+1}\Vert^4_{L^3}\leq C,\label{es-2}\\
		&\sum_{k=0}^n\tau\Vert\nabla \mu^{k+1}-c^k\nabla (c^{k+1}-\phi^k)\Vert^2+\sum_{k=0}^n\tau\Vert\mu^{k+1}\Vert_{W^{1,\frac65}}^\frac43\leq C.\label{es-3}
	\end{align}
	
	Furthermore, 
	it follows from \eqref{2.phi} that
	\begin{align}
		&\tau\left\Vert\frac{\delta_t\phi^{n+1}}\tau\right\Vert_{(H^1)'}^2 \leq C\tau\Vert\nabla \mu^{n+1}-c^n\nabla \big(c^{n+1}-\phi^n\big)\Vert^2,\label{basic-4}
	\end{align}
	and, by using $\Vert\nabla P_{\X_h}\eta\Vert_{L^p}\le C\Vert\nabla \eta\Vert_{L^p}$ (see \cite{CT87}) for some positive constant $C>0$ and any $p\in [2, \infty)$, it follows from \eqref{2.c} that, for any $\eta\in W^{1,4}(\Omega)\subset L^2(\Omega)$,
	\begin{align}
		&\tau^{-1}\vert\langle\delta_t c^{n+1},\eta\rangle\vert=\tau^{-1}\vert\langle\delta_t c^{n+1},P_{\X_h}\eta\rangle\vert\notag\\
		&\leq C\big(\Vert\nabla \mu^{n+1}-c^n\nabla (c^{n+1}-\phi^n)\Vert \Vert  c^n\Vert_{L^4}+\Vert\nabla \big(c^{n+1}-\phi^n\big)\Vert\big) \Vert\eta\Vert_{W^{1,4}(\Omega)},\notag
	\end{align} which yields
	\begin{align}\tau^{-1}\Vert \delta_t c^{n+1}\Vert _{(W^{1,4}(\Omega))'}\le C\Big(\Vert\nabla \mu^{n+1}-c^n\nabla (c^{n+1}-\phi^n)\Vert \Vert  c^n\Vert_{L^4}+\Vert\nabla \big(c^{n+1}-\phi^n\big)\Vert\Big). \notag\end{align}
	Using $(a+b)^{\frac43}\leq2^{\frac43}(a^{\frac43}+b^{\frac43})$ for all $a,b\geq0$ and the H\"older's inequality, we have
	\begin{align}
		\sum_{k=0}^{N-1}\tau\left\Vert \frac{\delta_t c^{k+1}}\tau\right\Vert _{(W^{1,4}(\Omega))'}^{\frac43}&\le C\tau\left(\sum_{k=0}^{N-1}\Vert\nabla \mu^{k+1}-c^k\nabla (c^{k+1}-\phi^k)\Vert^{\frac43} \Vert c^k\Vert_{L^4}^{\frac43}\right)\noN\\&+ C\tau\left(\sum_{k=0}^{N-1}\Vert\nabla \big(c^{k+1}-\phi^k\big)\Vert^{\frac43}\right)\nonumber\\
		&\le C\left(\sum_{k=0}^{N-1}\tau\Vert\nabla \mu^{k+1}-c^k\nabla (c^{k+1}-\phi^k)\Vert^2\right)^{\frac23} \left(\sum_{k=0}^{N-1}\tau\Vert c^k\Vert_{L^4}^4\right)^{\frac13}\noN\\&+C\left(\tau\sum_{k=0}^{N-1}\Vert\nabla \big(c^{k+1}-\phi^k\big)\Vert^2\right)^{\frac23}\left(\tau\sum_{k=0}^{N-1}1\right)^{\frac13}. \notag
	\end{align}
	By the Cauchy's inequality, and using the embedding inequality $\Vert v\Vert _{L^4}\le \Vert v\Vert _{L^2}^{\frac12}\Vert v\Vert _{H^1(\Omega)}^{\frac12}$ when $d=1$ and \eqref{basic-1}, we have
	\begin{align}
		\sum_{k=0}^{N-1}\tau\left\Vert \frac{\delta_t c^{k+1}}\tau\right\Vert _{(W^{1,4}(\Omega))'}^{\frac43}&\le C\left(\sum_{k=0}^{N-1}\tau\Vert\nabla \mu^{k+1}-c^k\nabla (c^{k+1}-\phi^k)\Vert^2\right)^{\frac23} \left(\sum_{k=0}^{N-1}\tau\Vert c^k\Vert_{L^2}^2\Vert c^k\Vert ^2_{H^1(\Omega)}\right)^{\frac13}\notag\\
		&\phantom{xx}+C\left(\tau\sum_{k=0}^{N-1}\Vert\nabla \big(c^{k+1}-\phi^k\big)\Vert^2\right)^{\frac23}\left(\sum_{k=0}^{N-1}\tau\right)^{\frac13}\le C<\infty.\label{ct-est}\end{align}
	
	\begin{rem}{If $d=2$, we can use the Gagliardo-Nirenberg inequality $\Vert c\Vert_{L^4}\le C\Vert c\Vert_{L^2}^{\frac12}\Vert \nabla c\Vert_{L^2}^{\frac12}$ to prove the bound \eqref{ct-est} in the above line. If $d=3$, we can use $\Vert c\Vert_{L^3}\le C\Vert c\Vert_{L^2}^{\frac12}\Vert c\Vert_{L^6}^{\frac12}$ and consider $\Vert \frac{\delta_t c^{k+1}}\tau\Vert _{(W^{1,6}(\Omega))'}^{\frac43}$ instead.}\end{rem}
	
	Therefore, by \eqref{basic-4} and \eqref{ct-est}, we have
	\begin{align}
		\tau\sum_{k=0}^n\left(\Vert\tau^{-1}{\delta_t\phi^{k+1}}\Vert_{(H^1)'}^2 + \Vert\tau^{-1}{\delta_t c^{k+1}}\Vert_{X'}^{\frac43}\right)\leq C.\label{de-st}
	\end{align}
	
	Using \eqref{basic-1}, \eqref{es-1}-\eqref{es-3} and \eqref{de-st}, we have the following estimates uniformly in $h>0$ and $\tau>0$
	\begin{align}
		&\Vert(\tilde\phi,\phi^\pm)\Vert_{L^\infty([0,T];H^1(\Omega))}^2+\frac1\tau\Vert\tilde\phi -\phi^\pm\Vert_{L^2([0,T];H^1(\Omega))}^2+\Vert\tilde\phi_t\Vert_{L^2([0,T];(H^{1}(\Omega))')}^2+\Vert\tilde c_t\Vert_{L^{4/3}([0,T];X')}^2\notag\\
		&+\Vert(\tilde c,c^+,c^-)\Vert_{L^\infty([0,T];L^2(\Omega))}^2+\Vert(\tilde c,c^+,c^-)\Vert_{L^2([0,T];H^1(\Omega))}^2+\Vert(\tilde c,c^+,c^-)\Vert_{L^4([0,T];L^3(\Omega))}^4\nonumber\\
		&+\frac1\tau\Vert\tilde c -c^\pm\Vert_{L^2([0,T];L^2(\Omega))}^2+\Vert\mu^+\Vert_{L^\frac43([0,T];W^{1,\frac65}(\Omega))}^\frac43\notag\\
		&+\Vert c^+-\phi^-\Vert_{L^2([0,T];H^1(\Omega))}^2+\Vert\nabla \mu^+-c^-\nabla (c^+-\phi^-)\Vert_{L^2([0,T];L^2(\Omega))}^2\leq C.\label{est-fin}
	\end{align}
	
	Next, we obtain some convergence results by the compactness theory. Firstly, it follows from bounds \eqref{est-fin} uniformly in $h>0$ and $\tau>0$ that there exist $\phi^*,\phi^{*\pm},c^*, c^{*\pm},\mu^*$ and $g^*$ such that $\phi^*=\phi^{*\pm},c^*=c^{*\pm}$, $(\phi^*,c^*,\mu^*)\in L^\infty(0,T;H^1(\Omega))\times L^2([0,T]; L^2(\Omega))\times L^\frac43([0,T];W^{1,\frac65}(\Omega))$, $(\partial_t\phi^*,\partial_tc^*)\in L^2([0,T]; (H^1(\Omega))')\times L^{4/3}([0,T]; X')$, and the convergence properties \eqref{cog-1}-\eqref{cog-3},\eqref{cog-5}-\eqref{cog-7}, \eqref{cog-9}-\eqref{cog-10} hold when $(h,\tau)\rightarrow 0$. Then, by the compact Sobolev's embedding $H^1\hookrightarrow\hookrightarrow L^s$ for any $s\in [2,6)$ when $d=1,2,3$ and Lions-Aubin lemma, the convergence properties  \eqref{cog-4} and \eqref{cog-8} hold when $(h,\tau)\rightarrow 0$. Also, by combining the strong convergence of $c^{-}$ in $L^2([0,T]; L^2(\Omega))$ and the weak convergence of $\nabla (c^+-\phi^-)$ and $c^-\nabla (c^+-\phi^-)$ in $L^2([0,T]; L^2(\Omega))$, and using \eqref{cog-10}, we have
	\begin{align}\nabla \mu^+-c^-\nabla (c^+-\phi^-)\rightarrow \nabla \mu^*-c^*\nabla (c^*-\phi^*) \text{ weakly }\text{ in } {L^\frac43([0,T];L^\frac65(\Omega))},\label{non-1}\end{align}
	which yields the desired convergence \eqref{cog-11}. Similarly, we can get another weak convergence property for the nonlinear term in the right-hand side of the equation \eqref{sou-2}
	\begin{align}c^-\big(\nabla \mu^+-c^-\nabla (c^+-\phi^-)\big)\rightarrow c^*\big(\nabla \mu^*-c^*\nabla (c^*-\phi^*)\big) \text{ weakly in } {L^\frac43([0,T];L^\frac65(\Omega))}.\label{non-2}\end{align}
	
	
	Finally, we prove $\phi^*,c^*,\mu^*$ is a solution of \eqref{1.phi} - \eqref{1.mu}. To do this, we rewrite the system \eqref{2.phi}-\eqref{2.mu} in $[t_n,t_{n+1}]$ in the following form
	\begin{align}
		\int_0^T\bigg(\left\langle \partial_t\tilde\phi,\xi_1 \right\rangle ,\chi_1\bigg) ds =& \int_0^T\bigg(-\left\langle \nabla \mu^+-c^-\nabla \big(c^+-\phi^-\big),\nabla \xi_1\right\rangle,\chi_1\bigg)ds,\label{app-1}\\
		\int_0^T\bigg(\left\langle \partial_t\tilde c,\eta_1\right\rangle,\zeta_1\bigg)ds =& \int_0^T\bigg(\left\langle c^-\nabla \mu^+-2(c^-)^2\nabla \big(c^{+}-\phi^-\big)- \nabla \big(c^{+}-\phi^-\big),\nabla \eta_1\right\rangle,\zeta_1\bigg)ds,\label{app-2} \\
		\int_0^T\bigg(\langle\mu^{+},\sigma_1\rangle,\upsilon_1\bigg)ds =& \int_0^T\bigg(\left\langle \nabla \phi^{+}, \nabla  \sigma_1\right\rangle,\upsilon_1\bigg)ds\notag\\
		&+\int_0^T\bigg(\left\langle \frac1{\varepsilon^2}\bigg( f(\phi^{-})+S(\phi^{+}-\phi^-)\bigg)- c^{+},\sigma_1\right\rangle,\upsilon_1\bigg)ds\label{app-3}
	\end{align} for any $(\xi_1,\eta_1, \sigma_1)\in (\X_h)^3$ and $(\chi_1(t), \zeta_1(t),\upsilon_1(t))\in (C^1[0,T])^3$.
	
	Since $C^\infty(\overline{\Omega})$ is dense in $H^1(\Omega)$ and also in $W^{1,6}(\Omega)$ and $C^\infty_0((0,T))$ is dense in $L^4(0,T)$, we just prove the system \eqref{sou-1}-\eqref{sou-3} holds for $(\xi(x),\eta(x),\sigma(x))\in (C^\infty(\overline{\Omega}))^3, (\chi(t), \zeta(t),\upsilon(t))\in (C^\infty_0((0,T)))^3$. Thus, for any $(\xi(x),\eta(x),\sigma(x))\in C^\infty(\overline{\Omega}), (\chi(t), \zeta(t),\upsilon(t))\in C^\infty_0((0,T))$, we take $\xi_1=P_{\X_h}\xi, \eta_1=P_{\X_h}\eta, \sigma_1=P_{\X_h}\sigma, (\chi_1(t), \zeta_1(t),\upsilon_1(t))=(\chi(t), \zeta(t),\upsilon(t))$ in the system \eqref{app-1}-\eqref{app-3}. Thus, it holds that $\Vert (P_{\X_h}\xi,P_{\X_h}\eta,P_{\X_h}\sigma)-(\xi,\eta,\sigma)\Vert _{W^{1,6}(\Omega)}\to 0$ as $h\to 0$.
	
	Noting that $f(\phi)$ is Lipschtz continuous, by using the convergence properties \eqref{cog-1}-\eqref{cog-11} and \eqref{non-1}-\eqref{non-2}, and by passing the limit $h\to 0$ and $\tau\to 0$ in the system \eqref{app-1}-\eqref{app-3}, we can get
	\begin{align*}
		\int_0^T\bigg(\left\langle \partial_t\phi^*,\xi \right\rangle,\chi\bigg) ds =& \int_0^T\bigg(-\left\langle \nabla \mu^*-c^*\nabla \big(c^*-\phi^*\big),\nabla \xi_1\right\rangle,\chi_1\bigg)ds,\\
		\int_0^T\bigg(\left\langle \partial_tc^*,\eta\right\rangle,\zeta\bigg)ds =& \int_0^T\bigg(\left\langle c^-\nabla \mu^*-(c^*)^2\nabla \big(c^{*}-\phi^*\big)- \nabla \big(c^{*}-\phi^*\big),\nabla \eta\right\rangle,\zeta\bigg)ds, \\
		\int_0^T\bigg(\langle\mu^{*},\sigma\rangle,\upsilon\bigg)ds =& \int_0^T\bigg(\left\langle \nabla \phi^{*}, \nabla \sigma\right\rangle,\upsilon\bigg)ds+\int_0^T\bigg(\left\langle \frac1{\varepsilon^2}f(\phi^{*})- c^{*},\sigma\right\rangle,\upsilon\bigg)ds
	\end{align*} for any $(\xi(x),\eta(x),\sigma(x))\in (C^\infty(\overline{\Omega}))^3, (\chi(t), \zeta(t),\upsilon(t))\in (C^\infty_0((0,T)))^3$. By the denseness, the system \eqref{sou-1}-\eqref{sou-3} holds in the stated sense.
	
	Moreover, by integrating by parts with respect to the time $t$ in the system \eqref{app-1}-\eqref{app-2} and passing the limit $h\to 0$ and $\tau\to 0$, we can obtain $(\phi^*(t=0)-\phi_0(x), \xi)=0$ holds for any $\xi\in H^1(\Omega)$ and $(c^*(t=0)- c_0(x), \eta)=0$ holds for any $\eta\in W^{1,4}(\Omega)$ when $d=1,2$ and $\eta\in W^{1,6}(\Omega)$ when $d=3$. This completes the proof of Theorem \ref{thm.cov}.
\end{proof}

\section{Numerical results}\label{sec-nr}
In this section, we use numerical experiments to verify the analysis in previous sections.  We use the standard double well potential \textcolor{blue}{$F(u) = \frac14(1-u^2)^2$} without truncation and $S=1$. Numerical experiments were performed using piecewise linear Lagrangian finite elements using Fenics. Good coordination was shown in the results.

In the first example, we test the temporal and spatial orders of convergence of our numerical scheme \eqref{2.phi}-\eqref{2.mu} at {$T=0.128$ in $(0,2\pi)^2$}. The initial conditions are chosen as {$\phi(0) = 0.05\cos(x)\cos(y)+0.3, c(0) = 0.05\cos(2x)\cos(2y)+0.5$}. Other parameters are chosen as {$\varepsilon = 0.3$}.

The temporal order of convergence is computed in the following way. First, a reference solution is calculated at {$\tau = 0.0005$ and $N_x\times N_y = 128\times128$} where $N_x$ denotes the number of intervals the domain is divided into. Then, with $N_x$ fixed, we change $\tau$ to compute the solution and a numerical error at {$T=128\times10^{-3}$}. Afterwards, the order of convergence is calculated accordingly. The spatial order is computed similarly. First, a reference solution is calculated at {$\tau = 10^{-3}$ and $N_x\times N_y = 256\times256$}. Then, with the time discretization is fixed, the spatial order is computed accordingly. The results are shown in Table \ref{table.1} and Table \ref{table.2}. The numerical simulation indicates that the temporal convergence rate is first order and the spatial convergence rate is second order.
\begin{table}
	\centering
	\begin{tabular}{ccccccc}
		\hline
		$\tau(\times10^{-3}$) & $\Vert \phi(T)-\phi_{\rm ref}(T)\Vert _{H^1}$ & rate & $\Vert c(T)-c_{\rm ref}(T)\Vert _{H^1}$ & rate & $\Vert \mu(T)-\mu_{\rm ref}(T)\Vert _{H^1}$ &rate\\
		\hline
		1 &  2.7729e-02 &      &  5.8986e-03 &      &  3.0377e-01 &      \\
		2 &  7.4492e-02 & 1.43 &  1.5870e-02 & 1.43 &  8.1141e-01 & 1.42 \\
		4 &  1.4313e-01 & 0.94 &  3.0615e-02 & 0.95 &  1.5453e+00 & 0.93 \\
		8 &  2.2257e-01 & 0.64 &  4.8135e-02 & 0.65 &  2.3811e+00 & 0.62 \\
		16 &  2.8043e-01 & 0.33 &  6.2216e-02 & 0.37 &  2.9998e+00 & 0.33 \\
		\hline
	\end{tabular}
	\caption{Temporal convergence rate.}
	\label{table.1}
	\centering
	\begin{tabular}{ccccccc}
		\hline
		$h(\times2\pi)$ & $\Vert \phi(T)-\phi_{\rm ref}(T)\Vert _{H^1}$ & rate & $\Vert c(T)-c_{\rm ref}(T)\Vert _{L^2}$ & rate & $\Vert \mu(T)-\mu_{\rm ref}(T)\Vert _{H^1}$ &rate\\
		\hline
		1/128 &  1.9815e-03 &      &  1.2892e-04 &      &  5.4552e-03 &      \\
		1/64 &  9.7946e-03 & 2.31 &  6.4359e-04 & 2.32 &  2.6888e-02 & 2.30 \\
		1/32 &  3.9258e-02 & 2.00 &  2.6982e-03 & 2.07 &  1.0698e-01 & 1.99 \\
		1/16 &  1.3022e-01 & 1.73 &  1.1068e-02 & 2.04 &  3.6181e-01 & 1.76 \\
		1/8 &  2.7451e-01 & 1.08 &  3.4138e-02 & 1.63 &  9.3435e-01 & 1.37 \\
		\hline
	\end{tabular}
	\caption{Spatial convergence rate.}
	\label{table.2}
\end{table}

In the second example, we test our numerical scheme's temporal and spatial convergence orders \eqref{2.phi}-\eqref{2.mu} concerning the norm defined in Theorem \ref{thm.od} in $(0,2\pi)$. The same initial conditions in the first experiments are used. In the numerical experiments, convergence rates of the first order in time and the second order in space can be observed. The observed second order convergence is probably attribute to the interpolation of the finer mesh reference solution onto the coarser mesh.
\begin{figure}[htp!]
	\centering
	\includegraphics[width=35mm]{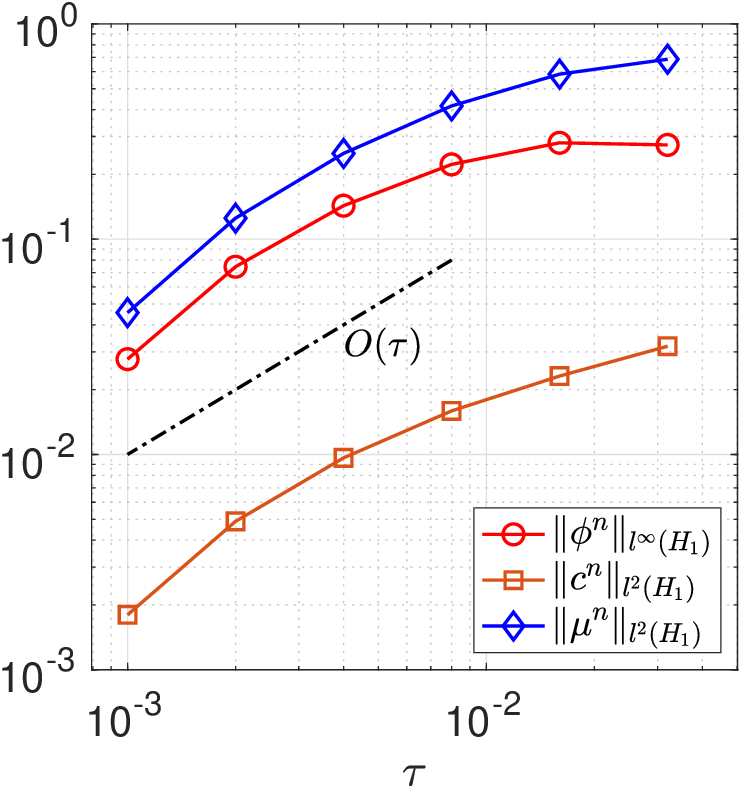}
	\includegraphics[width=35mm]{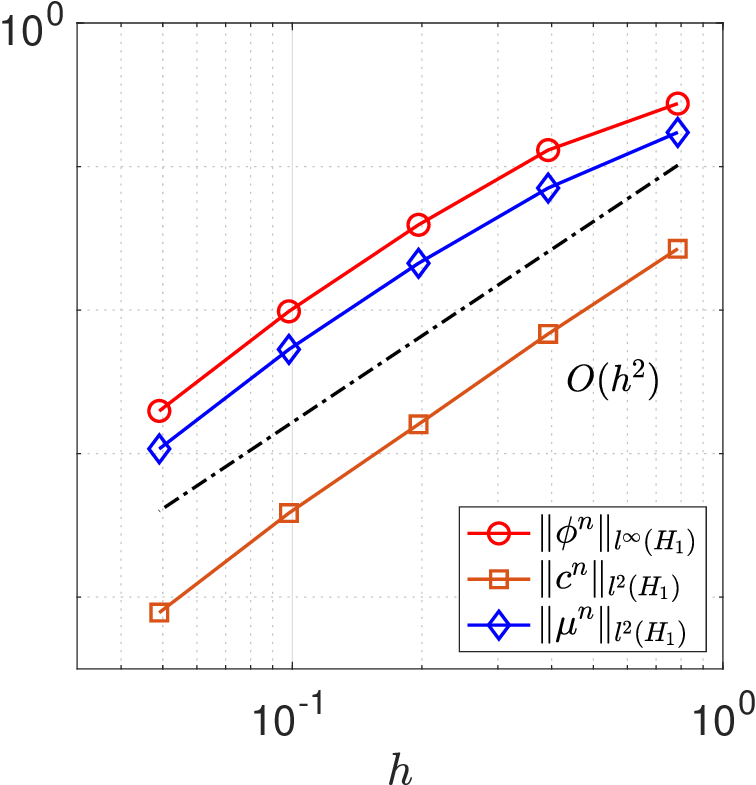}
	\caption{Disrete numerical errors for various values of the time step $\tau$ and spatial mesh size $h$.}
	\label{fig.1.tau}
\end{figure}

{In the third experiment, we present a morphological simulation with $\epsilon = 0.3$ and a constant mobility $g(c) = 0.01$. As $g(c)$ remains constant, the previous analysis still applies. In addition, this choice helps maintain the positivity of $c$ without significantly affecting the morphological dynamics. The numerical results is similar to those in \cite{RoFo08}.}
\begin{figure}[H]
	\centering
	
	\includegraphics[width=0.22\textwidth]{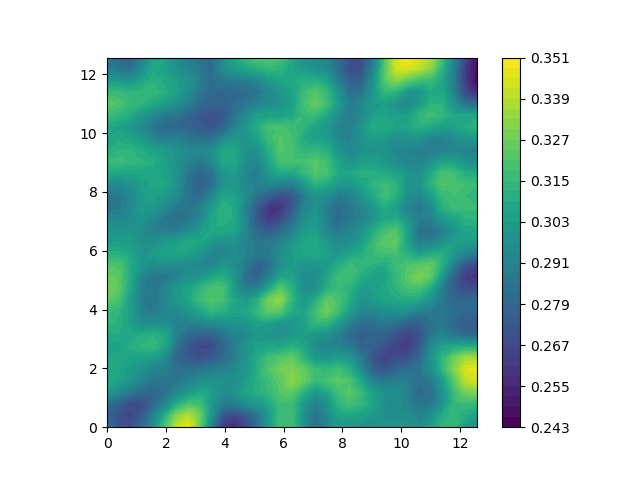} \hspace{0.5em}
	\includegraphics[width=0.22\textwidth]{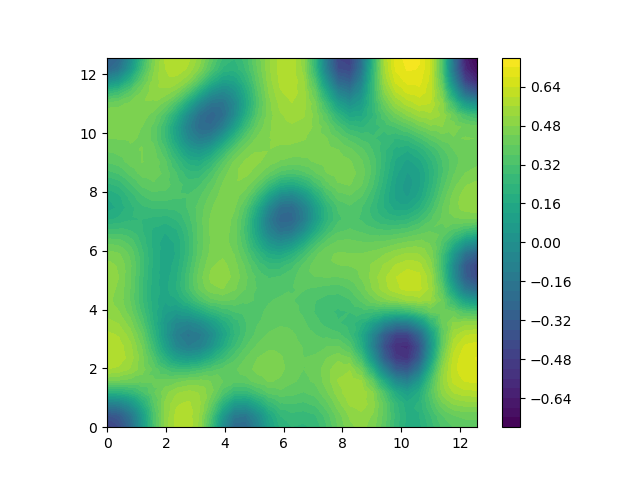} \hspace{0.5em}
	\includegraphics[width=0.22\textwidth]{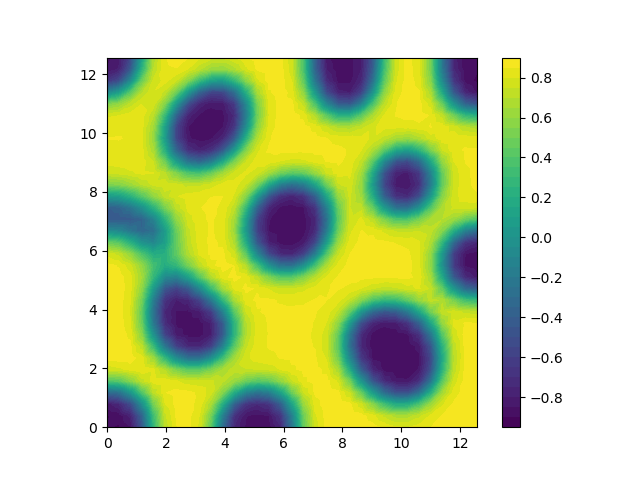} \hspace{0.5em}
	\includegraphics[width=0.22\textwidth]{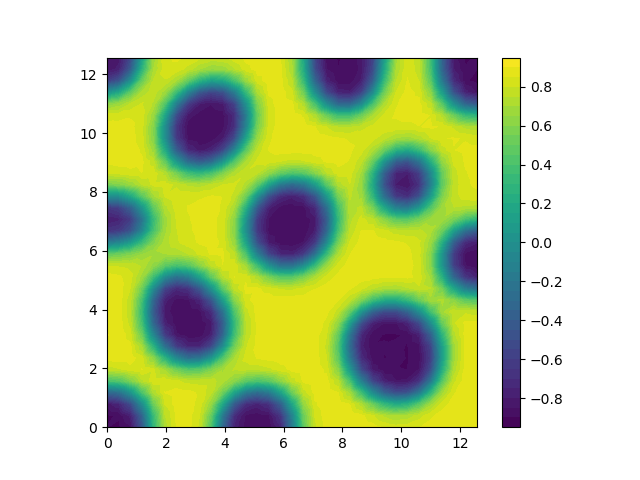}\\
	
	\vspace{1em} 
	
	\includegraphics[width=0.22\textwidth]{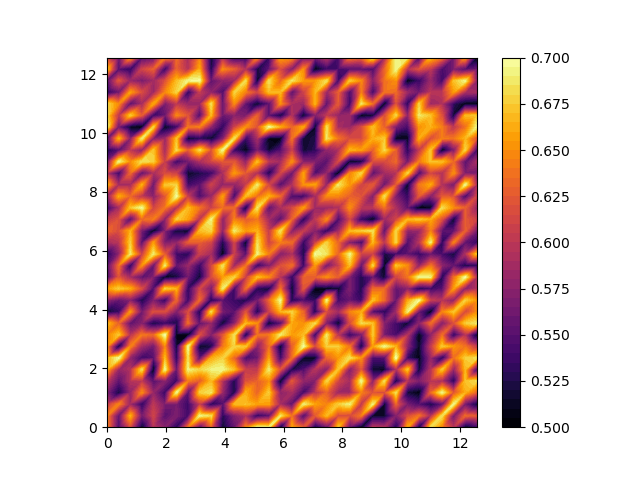} \hspace{0.5em}
	\includegraphics[width=0.22\textwidth]{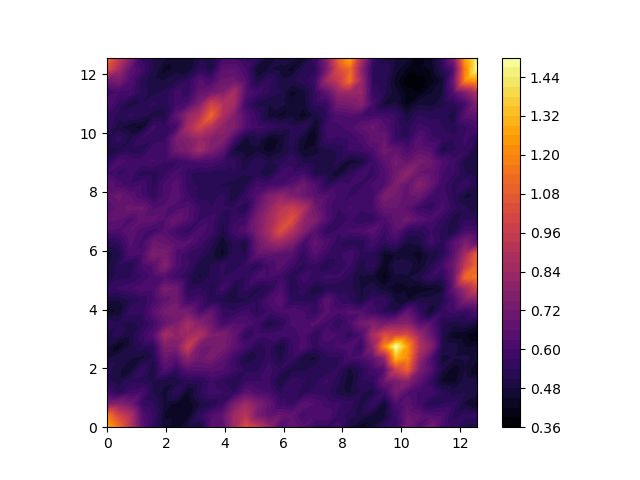} \hspace{0.5em}
	\includegraphics[width=0.22\textwidth]{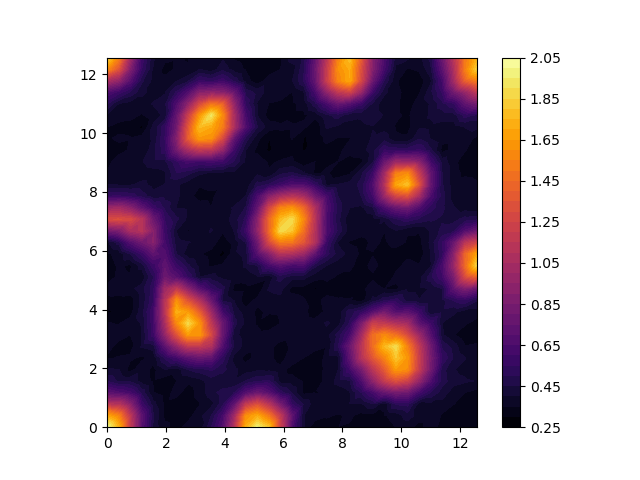} \hspace{0.5em}
	\includegraphics[width=0.22\textwidth]{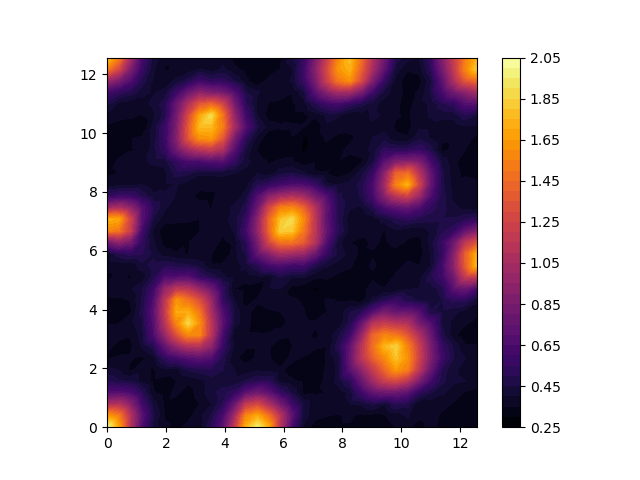}\\
	
	\caption{Evolution of $\phi$ (top row) and $c$ (bottom row) at selected time steps $t = 0.0, 0.4, 0.8, 1.0$.}
	\label{fig:phic_snapshots}
\end{figure}

\section{Conclusions}\label{sec-con}
In conclusion, this paper provides error estimates and proves the convergence of a stabilized numerical scheme for the Cahn--Hilliard(CH) cross-diffusion model. The energy stability and the existence of the numerical solution are proved as well. A technique is used to overcome the difficulty caused by the nonlinear cross-diffusion matrix of the system: the analysis of the error estimates and the convergence rates are based on the $\sum_{k=0}^n\tau\|\nabla(\cdot)\|_{L^{\frac65}}^{\frac43}$ norm estimate of the chemical potentials instead of the $\sum_{k=0}^n\tau\|\nabla(\cdot)\|_{L^{2}}^{2}$ norm estimate, which is different from the numerical schemes for the CH equation in literature. This study explores the link between the cross-diffusion model and the CH equation and demonstrates the numerical scheme's feasibility. Additionally, the numerical results verify the theoretical findings.
Future work will focus on the degenerate model, which is a more practical and challenging scenario. This investigation will enhance understanding of the system's behavior and develop more robust modeling and analyzing techniques.

\section*{Data availability}
The author does not analyze or generate any datasets, as the research is theoretical and focuses on mathematical approaches.

\section*{Declaration}
\textbf{Competing Interests} The author declares no competing interests relevant to the contents of this article.



\bibliographystyle{ieeetr}

\end{document}